\documentclass[12pt]{amsart}
\usepackage{amscd}
\usepackage{amssymb}
\usepackage{graphicx}
\usepackage{color}
\usepackage[centering,text={15.5cm,22cm},
		marginparwidth=20mm]{geometry}
\usepackage{hyperref}

\usepackage{mathrsfs}

\usepackage[all]{xy}

\newtheorem{theorem}{Theorem}[section]

\newtheorem{assumptions}[theorem]{Assumptions}

\newtheorem{lemma}[theorem]{Lemma}
\newtheorem{thmconst}[theorem]{Theorem-Construction}

\newtheorem{proposition}[theorem]{Proposition}
\newtheorem{corollary}[theorem]{Corollary}

\theoremstyle{definition}
\newtheorem{definition}[theorem]{Definition}
\newtheorem{definition-lemma}[theorem]{Definition-Lemma}

\newtheorem{construction}[theorem]{Construction}

\newtheorem{example}[theorem]{Example}

\theoremstyle{remark}
\newtheorem{remark}[theorem]{Remark}

\numberwithin{equation}{section}
\numberwithin{figure}{section}

\newcommand{\ZZ} {\mathbb{Z}}

\newcommand{\RR} {\mathbb{R}}
\newcommand{\bR} {\RR}

\newcommand{\PP} {\mathbb{P}}

\newcommand{\GG} {\mathbb{G}}

\newcommand {\shO}  {\mathcal{O}}

\newcommand {\meij} {{\{E_{ij}\}}}
\newcommand {\Adm} {\operatorname{Adm}}
\newcommand {\mfij} {\{F_{ij}\}}

\newcommand {\Aut}  {\operatorname{Aut}}

\newcommand {\can}  {\mathrm{can}}

\newcommand {\Gm} {\GG_m}

\newcommand {\Hom}  {\operatorname{Hom}}

\newcommand {\id}  {\operatorname{id}}

\newcommand {\Int}  {\operatorname{Int}}

\newcommand {\Nef}  {\operatorname{{Nef}}}

\newcommand {\Pic}  {\operatorname{Pic}}

\newcommand {\NE}   {\operatorname{NE}}

\def\mapright#1{\smash{
  \mathop{\longrightarrow}\limits^{#1}}}

\def\oL{\overline{L}}

\def\bP{\mathbb P}

\def\bN{\mathbb N}

\def\Nef{\operatorname{Nef}}

\def\bF{\mathbb F}

\def\Aut{\operatorname{Aut}}

\def\tY{\tilde{Y}}

\def\cO{\Cal O}

\def\bZ{\mathbb Z}
\def\bC{\mathbb C}

\def\bG{\mathbb G}
\def\bR{\mathbb R}

\def\oY{\bar{Y}}

\def\oD{\bar{D}}
\def\op{\bar{p}}

\def\cD{\Cal D}

\def\cE{\Cal E}
\def\cF{\Cal F}
\def\tS{\tilde S}

\def\oS{\overline{S}}

\def\tS{\tilde{S}}

\def\oN{\overline{N}}

\def\cY{\Cal Y}
\def\tcY{\tilde{\cY}}
\def\ocY{\overline{\cY}}

\def\ocD{\overline{\Cal D}}

\def\cD{\Cal D}
\def\cX{\Cal X}

\def\cM{\Cal M}

\def\cE{\Cal E}

\def\cO{\Cal O}

\def\rk{\operatorname{rk}}

\def\Pic{\operatorname{Pic}}

\def\Hom{\operatorname{Hom}}
\def\Cal{\mathcal}

\def\Efi#1#2#3#4#5{\displaystyle
#1\!\!-\!\!#2
\!\!-\!\!#3
\!\!-\!\!#4
\hskip-24.2pt\lower4.5pt\hbox{${\scriptstyle|}
\hskip-3.35pt\lower6pt\hbox{$#5$}$}}

\def\Evia#1#2#3#4#5{\displaystyle
#1\!\!-\!\!#2
\!\!-\!\!#3
\hskip-24.2pt\lower4.5pt\hbox{${\scriptstyle|}
\hskip-3.35pt\lower6pt\hbox{$#4\!\!-\!\!\!-\!\!\!-\!\!$}$\hskip2.3pt${\scriptstyle|}
\hskip-3.35pt\lower6pt\hbox{$#5$}$}}

\def\Ezia#1#2#3#4{\displaystyle
#1\!\!-\!\!#2
\hskip-14.8pt\lower4.5pt\hbox{${\scriptstyle|}
\hskip-3.35pt\lower6pt\hbox{$#3\!\!-\!\!$}$\hskip2.3pt${\scriptstyle|}
\hskip-3.35pt\lower6pt\hbox{$#4$}$}}

\def\Efia#1#2#3#4#5#6{\displaystyle
#1\!\!-\!\!#2
\!\!-\!\!#3
\!\!-\!\!#4
\hskip-24.2pt\lower4.5pt\hbox{${\scriptstyle|}
\hskip-3.35pt\lower6pt\hbox{$#5$}$\hskip5.7pt${\scriptstyle|}
\hskip-3.35pt\lower6pt\hbox{$#6$}$}}

\def\Esi#1#2#3#4#5#6{\displaystyle
#1\!\!-\!\!#2
\!\!-\!\!#3
\!\!-\!\!#4\!\!-\!\!#5
\hskip-24.2pt\lower4.5pt\hbox{${\scriptstyle|}
\hskip-3.35pt\lower6pt\hbox{$#6$
\lower3pt\hbox{\ }}$}}

\def\Esia#1#2#3#4#5#6#7{\displaystyle

#1\!\!-\!\!#2
\!\!-\!\!#3
\!\!-\!\!#4\!\!-\!\!#5
\hskip-24.2pt\lower4.5pt\hbox{${\scriptstyle|}
\hskip-3.35pt\lower6pt\hbox{$#6$\hskip-3.8pt\lower4.5pt\hbox{${\scriptstyle|}
\hskip-3.35pt\lower6pt\hbox{$#7$}$}}
\lower3pt\hbox{\ }$}}

\def\Ese#1#2#3#4#5#6#7{\displaystyle
#1\!\!-\!\!#2
\!\!-\!\!#3
\!\!-\!\!#4\!\!-\!\!#5\!\!-\!\!#6
\hskip-33.6pt\lower4.5pt\hbox{${\scriptstyle|}
\hskip-3.35pt\lower6pt\hbox{$#7$
\lower3pt\hbox{\ }
}$}}

\def\Esea#1#2#3#4#5#6#7#8{\displaystyle
#1\!\!-\!\!#2
\!\!-\!\!#3
\!\!-\!\!#4\!\!-\!\!#5\!\!-\!\!#6\!\!-\!\!#7
\hskip-33.6pt\lower4.5pt\hbox{${\scriptstyle|}
\hskip-3.35pt\lower6pt\hbox{$#8$
\lower3pt\hbox{\ }
}$}}

\def\Eei#1#2#3#4#5#6#7#8{\displaystyle
#1\!\!-\!\!#2
\!\!-\!\!#3
\!\!-\!\!#4\!\!-\!\!#5\!\!-\!\!#6\!\!-\!\!#7
\hskip-43.2pt\lower4.5pt\hbox{${\scriptstyle|}
\hskip-3.35pt\lower6pt\hbox{$#8$
\lower3pt\hbox{\ }
}$}}

\def\Eeia#1#2#3#4#5#6#7#8#9{{\displaystyle
#1\!\!-\!\!#2
\!\!-\!\!#3
\!\!-\!\!#4\!\!-\!\!#5\!\!-\!\!#6\!\!-\!\!#7\!\!-\!\!#8
\hskip-52.2pt\lower4.5pt\hbox{${\scriptstyle|}
\hskip-3.35pt\lower6pt\hbox{$#9$
\lower3pt\hbox{\ }
}$}}}

\def\os{\overline{S}}

\def\oH{\bar{H}}

\def\tS{\tilde{S}}
\def\ts{\tS}
\def\ts7{\tilde{S}_7}

\def\bP{\mathbb P}
\def\bF{\mathbb F}

\def\Aut{\operatorname{Aut}}

\def\Hom{\operatorname{Hom}}

\def\tY{\tilde{Y}}

\def\cO{\Cal O}

\def\cD{\Cal D}

\def\Pic{\operatorname{Pic}}

\def\bZ{\mathbb Z}
\def\bC{\mathbb C}

\def\bG{\mathbb G}
\def\bR{\mathbb R}

\def\cD{\Cal D}

\def\cE{\Cal E}
\def\cF{\Cal F}
\def\tS{\tilde S}

\def\oS{\overline{S}}

\def\tS{\tilde{S}}

\def\oN{\overline{N}}

\def\cY{\Cal Y}
\def\cX{\Cal X}

\def\cM{\Cal M}

\def\cE{\Cal E}

\def\Cal{\mathcal}
\def\Pic{\operatorname{Pic}}

\def\oN{\overline{N}}
\def\os7p{\oS_7'}
\def\os7{\oS_7}
\def\on6{\oN_6}
\def\n6{\oN_6}

\DeclareMathOperator{\MW}{MW}
\DeclareMathOperator{\Fix}{Fix}
\DeclareMathOperator{\Hodge}{Hodge}
\DeclareMathOperator{\Def}{Def}
\def\oT{\bar{T}}

\title[Moduli of surfaces with an anti-canonical cycle]
{Moduli of surfaces with an anti-canonical cycle}

\author{Mark Gross} 
\address{DPMMS, Centre for Mathematical Sciences,
Wilberforce Road, Cambridge CB3 0WB}
\email{mgross@dpmms.cam.ac.uk}
\author{Paul Hacking}
\address{Department of Mathematics and Statistics, Lederle Graduate
Research Tower, University of Massachusetts, Amherst, MA 01003-9305}
\email{hacking@math.umass.edu}

\author{Sean Keel}
\address{Department of Mathematics, 1 University Station C1200, Austin,
TX 78712-0257}
\email{keel@math.utexas.edu}
\def\mydate{\ifcase\month \or January\or February\or March\or
April\or May\or June\or July\or August\or September\or October\or 
November\or December\fi \space\number\day,\space\number\year}

\begin{document}

\begin{abstract}
We prove a global Torelli theorem for pairs $(Y,D)$ where $Y$ is a 
smooth projective
rational surface and $D\in|-K_Y|$ is a cycle of rational curves,
as conjectured by Friedman in 1984.
In addition, we construct natural universal families for such pairs.
\end{abstract}

\maketitle
\tableofcontents
\bigskip


\section{Introduction}

We work throughout over the field $k=\bC$. We work in the algebraic category unless explicitly stated otherwise.

\begin{definition} \label{def2dlp} A \emph{Looijenga pair} $(Y,D)$ 
is a smooth projective surface $Y$ together with a connected singular
nodal curve $D \in |-K_Y|$. 
Note $p_a(D) =1$ by adjunction, so $D$ is either an irreducible rational curve with a single node, or a cycle of smooth rational curves.  
We fix an \emph{orientation} of the cycle $D$, that is, a choice of generator of $H_1(D,\bZ) \cong \bZ$, and an ordering $D=D_1+\cdots+D_n$ of the irreducible components of $D$ compatible with the orientation.
\end{definition}

By an \emph{isomorphism of Looijenga pairs} $(Y^1,D^1)$, $(Y^2,D^2)$ we mean an isomorphism $f \colon Y^1 \rightarrow Y^2$ such that $f(D^1_i)=D^2_i$ for each $i=1,\ldots,n$ and $f$ is compatible with the orientations of $D^1$ and $D^2$. We write $\Aut(Y,D)$ for the group of automorphisms of a Looijenga pair $(Y,D)$ in this sense.

By the birational classification of surfaces, $Y$ in Definition \ref{def2dlp}
is necessarily rational. 

Looijenga pairs were introduced in \cite{L81} as natural log analogs of 
K3 surfaces.  Looijenga studied the cases $n \leq 5$ in detail.
Here we consider moduli of Looijenga pairs with no restriction on $n$.
We prove the global Torelli
Theorem, conjectured by Friedman in \cite{F84}, see Theorem \ref{torelliI}. 
We construct natural universal families (\S\ref{univfam}), give a precise description of the moduli stack of Looijenga pairs (Theorem~\ref{modulistacks}) and identify the monodromy group (Theorem~\ref{strong_monodromy}).
 
The motivation for studying Looijenga pairs comes from several directions.
Our initial interest arose from the construction of \cite{GHKI}. There
we construct a mirror family to any Looijenga pair $(Y,D)$. If the intersection
matrix of the components of $D$ is not negative semi-definite, 
then our construction
yields an algebraic family. We call this the \emph{positive} case.
In the sequel \cite{GHKII} to that work, we will
apply the Torelli theorem to show that in the positive case 
the mirror family is the universal
family of Looijenga pairs constructed here.  
This has a striking
consequence: our construction of the mirror family endows the fibres  
with a canonical
basis of functions. We call elements of this basis theta functions,
as a related construction yields theta functions on abelian
varieties. Realizing this as the  
universal family now
endows each affine surface $U = Y \setminus D$ in the family
with canonical theta  
functions. Though these
include some of the most classical objects in geometry, e.g., $(Y,D)$
could be a cubic surface
with a triangle of lines, in which case $U$ is what Cayley called an  
affine cubic surface, we
do not believe this canonical basis has been previously observed, or  
even conjectured.

A second application of the universal families is given in \cite{GHKIII}, where
we show that Looijenga pairs are closely related to rank $2$ cluster
varieties, and realize the Fock-Goncharov fibration of the cluster $\cX$-variety (in
the rank $2$ case) as a natural 
quotient of our universal families. (See \cite{FG}, \cite{FZ} for the definitions of cluster
varieties.) 
In any event,
Looijenga pairs appear in a number of other settings,
such as
the study of degenerations of $K3$ surfaces: the central fibres for maximal
degenerations, type III in Kulikov's classification, are normal crossing
unions of such pairs. 

Looijenga pairs have an elementary construction:

\begin{definition} \label{defccy} Let $(\oY,\oD)$ be a smooth projective toric
surface, where $\oD := \oY \setminus \bG_m^2$ is the toric boundary, i.e., the
union of toric divisors of $Y$. Let 
$\pi: Y \to \oY$ be the blowup at some number of smooth points (with 
infinitely near points allowed) of $\oD$. Let $D \subset Y$ be the strict
transform of $D$. Then $(Y,D)$ is a Looijenga pair, and we call 
$\pi: Y \to \oY$ a {\it toric model} for $(Y,D)$. 
\end{definition}

Essentially all Looijenga pairs arise in this way (i.e., have a toric
model). Indeed,
define a \emph{simple
toric blowup} $(Y',D') \to (Y,D)$ to be the blowup at a node of $D$,
with $D'$ the reduced inverse image of $D$. 
A {\it toric blowup}
is a composition of simple toric blowups. Note $(Y',D')$ is 
again a Looijenga pair, and the log Calabi-Yau is the same, i.e.,
$Y' \setminus D' = Y \setminus D$. 
We then have (see \cite{GHKI}, Prop.~1.19) the easy fact:

\begin{lemma} \label{tmexists} Given a Looijenga pair $(Y,D)$ 
there is a toric blowup $(Y',D')$ such
that $(Y',D')$ has a toric model. 
\end{lemma}

For any question we consider, passing to
a toric blowup will be at most a notational inconvenience.

To give a precise statement of our results, we first give a number of
basic definitions.

\begin{definition} Let $(Y,D)$ be a Looijenga pair.
\begin{enumerate}
\item
A curve $C \subset Y$ is {\it interior} if no irreducible
component of $C$ is contained in $D$. 
\item
An {\it internal $(-2)$-curve} means a smooth rational curve of
self-intersection $-2$ disjoint from $D$. 
\item
$(Y,D)$ is
{\it generic} if it has no internal $(-2)$-curves. 
\end{enumerate}
\end{definition}

Any Looijenga pair is deformation equivalent to a generic pair, see Proposition~\ref{generics}.
Note that, by adjunction, any irreducible interior curve with negative self-intersection number is either 
a $(-1)$-curve meeting $D$ transversely at a single smooth point,
or an internal $(-2)$-curve. Note also that if $(Y,D)$ is generic and $\pi:Y\to\oY$ is a toric model, 
then the blown up points are necessarily distinct (as opposed to infinitely near).

We next consider the notion of periods of Looijenga pairs.
We first note (see Lemma~\ref{pic0D}) 
that the orientation of $D$ determines a canonical identification
$\bG_m = \Pic^0(D)$, where the latter is the connected component
of the identity of $\Pic(D)$. 

\begin{definition}
Let 
$$
D^{\perp} := \{\alpha \in \Pic(Y)\,|\, \alpha \cdot [D_i] =0 \text{ for all }
i\}.
$$
Restriction of line bundles determines a canonical homomorphism
\begin{equation}
\label{unmarkedperiod}
\phi_Y: D^{\perp} \to \Pic^0(D) = \bG_m,
\quad L\mapsto L|_D.
\end{equation}
The homomorphism $\phi_Y \in T_{D^{\perp}}:= \Hom(D^{\perp},\bG_m)$
is called the 
\emph{period point} of $Y$.
\end{definition}

Note $Y \setminus D$ comes
with a canonical (up to scaling) nowhere-vanishing $2$-form, $\omega$, with simple poles along $D$.
One can show that 
$\phi_Y$ is equivalent to the data of periods of $\omega$ over cycles in
$H_2(Y \setminus D,\bZ)$, see \cite{F84}. This motivates the term
``period.''

As well as the notion of periods, we also need the following additional
notions to state the Torelli theorem.

\begin{definition} \label{rootsdef}
Let $(Y,D)$ be a Looijenga pair.
\begin{enumerate}
\item
The {\it roots} $\Phi \subset \Pic(Y)$
are those classes in 
$D^{\perp} \subset \Pic(Y)$ with square $-2$ which 
are realized by an internal 
$(-2)$-curve $C$ on a deformation equivalent pair $(Y',D')$.
More precisely, 
there is a family $(\cY,\cD)/S$, a path $\gamma \colon [0,1] \rightarrow S$,
and identifications  $(Y,D)=(\cY_{\gamma(0)},\cD_{\gamma(0)})$, 
$(Y',D')=(\cY_{\gamma(1)},\cD_{\gamma(1)})$, such that the isomorphism
$$H^2(Y',\bZ) \rightarrow H^2(Y,\bZ)$$ 
induced by parallel transport along $\gamma$ sends $[C]$ to $\alpha$.
\item Let $\Delta_Y \subset \Pic(Y)$ be the
set of classes of internal $(-2)$-curves.
\item Let $\Phi_Y \subset \Phi \subset \Pic(Y)$
be the subset of roots, $\alpha$, with $\phi_{Y}(\alpha) = 1$. Note that
$\Delta_Y \subset \Phi_Y\subset \Phi$. 

\item
Let $W \subset \Aut(\Pic(Y))$ be the subgroup 
generated by the reflections
$$
s_{\alpha} :\Pic(Y) \to \Pic(Y),\quad 
\beta \mapsto \beta + \langle\alpha,\beta\rangle
\alpha
$$
for $\alpha \in \Phi$. Let $W_Y \subset W$ be the subgroup
generated by $s_{\alpha}$ with $\alpha \in \Delta_Y$. 
\end{enumerate}
\end{definition}

It is clear from the definitions that $\Phi$ is invariant under parallel transport,
and 
$\Delta_Y,\Phi_Y,\Phi$ are all
invariant under $\Aut(Y,D)$. 
Further, the sets $\Phi$, $\Phi_Y$, $W$, $W_Y$ are easily seen to be 
invariant under toric blowup. Indeed, let $\tau:(Y',D')\rightarrow (Y,D)$
be a blow-up of a node of $D$. Then under pull-back $\tau^*$ of divisors,
$D^{\perp}$ is isomorphic to $(D')^{\perp}$ as lattices.

We will show that $\Phi_Y = W_Y \cdot \Delta_Y$, see Proposition~\ref{simplerootsgenerate}.

When $n \leq 5$ and the intersection matrix $(D_i \cdot D_j)$ is negative semidefinite, $\Phi$ contains a natural {\it root basis},
which is central to much of Looijenga's analysis. No such
basis exists in general. 

\begin{definition} \label{cpdef}
Let $(Y,D)$ be a Looijenga pair.
\begin{enumerate}
\item The cone
$\{x \in \Pic(Y)_{\bR} \ | \ x^2 > 0\}$ has two connected components.
Let $C^+$ be the connected component containing all the ample classes.
\item
For a given ample $H$ let $\tilde{\cM} \subset \Pic(Y)$ be the collection
of classes $E$ with $E^2 = K_Y\cdot E = -1$,  and $E \cdot H > 0$.
Note $\tilde{\cM}$ is independent of $H$, see Lemma \ref{C++lemma}. 
Let $C^{++} \subset C^{+}$ be the subcone defined by the
inequalities $x \cdot E \geq 0$ for all $E \in \tilde{\cM}$.
\item Let $C^{++}_D \subset C^{++}$ be the subcone 
where additionally $x \cdot [D_i] \geq 0$ for all $i$.
\end{enumerate}
 
\end{definition}

By Lemma \ref{C++lemma},
$C^+,C^{++},C^{++}_D$ and $\tilde{\cM}$ are all independent of 
deformation of Looijenga pairs (i.e., preserved by parallel transport).

Our main result is then:

\begin{theorem}(Torelli Theorem) \label{torelliI}
Let $(Y_1,D),(Y_2,D)$ be Looijenga pairs and 
let 
\[
\mu \colon \Pic(Y_1) \to \Pic(Y_2)
\]
be an isomorphism of lattices. 

{\bf Global Torelli:} $\mu = f^*$ for an isomorphism of 
pairs 
$f:(Y_2,D) \to (Y_1,D)$ iff all the following hold:
\begin{enumerate}
\item $\mu([D_i]) = [D_i]$ for all $i$. 
\item $\mu(C^{++})=C^{++}$.
\item $\mu(\Delta_{Y_1}) = \Delta_{Y_2}$.
\item $\phi_{Y_2} \circ \mu = \phi_{Y_1}$.
\end{enumerate}
If $f$ exists, the possibilities are a torsor for 
$\Hom(N',\bG_m)$ where $N'$ is the cokernel of the map
\[
\Pic(Y)\rightarrow\ZZ^n,\quad L\mapsto (L\cdot D_i)_{1\le i\le n}.
\]

{\bf Weak Torelli}:
There is an element $g$ in the Weyl group $W_{Y_1}$ such that
$\mu \circ g = f^*$ for an isomorphism of pairs
$f:(Y_2,D)\rightarrow (Y_1,D)$ iff $\mu$ satisfies conditions
(1),(2), and (4). If $g$ exists, it is
unique. 
\end{theorem}

\begin{remark}\label{torellinefcone}
We show that for a Looijenga pair $(Y,D)$ the nef cone $\Nef(Y)$ is the subcone of $\overline{C^{++}_D}$ defined by $x \cdot \alpha \ge 0$ for all $\alpha \in \Delta_Y$. See Lemma~\ref{Nefcones}.
Thus the global Torelli theorem can be restated as follows: Given Looijenga pairs $(Y_1,D)$ and $(Y_2,D)$ and an isomorphism of lattices $\mu \colon \Pic(Y_1) \rightarrow \Pic(Y_2)$, there is an isomorphism $f \colon (Y_2,D) \rightarrow (Y_1,D)$ of Looijenga pairs 
such that $\mu=f^*$ iff $\mu(\Nef(Y_1))=\Nef(Y_2)$ and $\mu([D_i])=[D_i]$ for each $i$.
\end{remark}

\begin{remark} \label{screwup} In a preliminary version of this note we claimed the Torelli
theorem with (2) replaced by the conditions $\mu(C^+) = C^+$ and $\mu(\Phi)=\Phi$.
R. Friedman showed us counterexamples to this statement \cite{F13}.
We note the weaker condition $\mu(C^+)=C^+$ is sufficient if $D$ supports a divisor of positive square, 
or if $\mu(H)$ is ample for some ample $H$, as either condition is easily seen to imply $\tilde{\cM}$, and thus $C^{++}$, is preserved. 
In \cite{F13} Friedman gives various sufficient conditions under which (2) may be replaced by the conditions $\mu(C^+) = C^+$ and $\mu(\Phi)=\Phi$
(all have the flavor of guaranteeing that $\Phi$ is sufficiently big).
\end{remark} 

The proof of the global Torelli theorem is carried out in \S\ref{torellisection}. The key point there is the notion
of a marked Looijenga pair and periods for marked Looijenga pairs.

\begin{definition}
\label{markingdef}
Let $(Y,D)$ be a Looijenga pair.
\begin{enumerate}
\item
A \emph{marking of $D$} is a choice of points
$p_i \in D_i^o$ for each $i$, where $D_i^o$ denotes the intersection of $D_i$ with the smooth locus of $D$.
This is equivalent to the choice of an isomorphism
$i: D^{\can} \to D$
of $D$ with a fixed cycle of rational curves $D^{\can}$. 
The possible markings of $D$ are a torsor for $\Aut^0(D)=\bG_m^n$, 
the connected
component of the identity of $\Aut(D)$.
\item
Fix $(Y_0,D)$ generic. A \emph{marking of $\Pic(Y)$} is an isomorphism
of lattices $\mu \colon \Pic(Y_0) \to \Pic(Y)$ such that $\mu([D_i])=[D_i]$ for each $i$ and $\mu(C^{++})=C^{++}$. 
\item
Markings $p_i,\mu$ determine a 
\emph{marked period point}:
\[
\phi_{((Y,D),p_i,\mu)} \in T_{Y_0}:= \Hom(\Pic(Y_0),\bG_m)
\]
by
\begin{equation} \label{basiceq}
\phi(L) := (\mu(L)|_D)^{-1} \otimes \cO_D\big(\sum (L \cdot D_i) p_i\big)
\in \Pic^0(D) = \bG_m.
\end{equation}
\end{enumerate}
\end{definition}

The global Torelli theorem is proved by first showing that given
a toric model for $(Y,D)$, the marked period point determines the location
of the blowups, and hence determines $Y$: this is essentially the
content of Proposition \ref{closetotorelli}. A bit more work leads
to the global Torelli theorem.

\bigskip

\emph{Acknowledgments}.
We received a great deal of inspiration from
numerous extended conversations with Robert Friedman, and
from ideas in his unpublished note \cite{F84}. Our original statement
of the Torelli theorem was false, see Remark \ref{screwup}. Friedman set us 
straight. We also had many
very helpful conversations with Eduard Looijenga, and the referee provided
numerous helpful comments. The first author
was partially supported by NSF grant DMS-1105871, the second by 
NSF grants DMS-0968824 and DMS-1201439, and the third by NSF grant DMS-0854747.


\section{The global Torelli Theorem} \label{torellisection}

\begin{lemma} \label{pic0D} 
Let $D$ be a cycle of $n$ rational curves, with cyclic ordering
of the components. This cyclic ordering induces:
\begin{enumerate}
\item An identification
$\Pic^0(D)=\Gm$, where the former is the group of
numerically trivial line bundles.
\item An identification $\Aut^0(D)=(\Gm)^n$, where the former
is the identity component of the automorphism group of $D$.
\end{enumerate}
\end{lemma}

\begin{proof} For (1), the fact that there is an abstract isomorphism
$\Gm\cong\Pic^0(D)$ is well-known, and the automorphism group of $\Gm$
as a group is
$\{1,-1\}$, so there are only two choices of identification.
Here is an explicit construction of an identification determined
by the orientation, which will be used throughout.
We assume $n \ge 3$, leaving the straightforward modifications for $n=1,2$ to the reader.
For $L \in \Pic^0(D)$, there is a nowhere-vanishing section 
$\sigma_i \in \Gamma(L|_{D_i})$. 
Let $\lambda_i := \sigma_{i+1}(p_{i,i+1})/\sigma_i(p_{i,i+1}) \in \bG_m$,
where $p_{i,i+1} := D_i \cap D_{i+1}$. 
Obviously  
$\lambda(L) := \prod_i \lambda_i$ is independent of the choice of 
$\sigma_i$. The map $L \mapsto \lambda(L)$ gives the canonical isomorphism. 

For (2), let $(x_i,y_i)$ be the homogeneous 
coordinates on $D_i$ with $x_i=0$ being
the point $D_{i-1}\cap D_i$. Then we take the $i^{th}$ copy of $\Gm$
to act on $D_i$ by $(x_i,y_i)\mapsto (x_i,\lambda y_i)$ for $\lambda\in\Gm$.
The $i^{th}$ copy of $\Gm$ acts trivially on $D_j$ for $j\not=i$.
\end{proof} 

Recall from Definition~\ref{defccy} the notion of a toric model of a Looijenga pair.

\begin{definition}
\label{exconfdef}
An {\it exceptional
configuration} for generic $(Y,D)$ means an ordered collection
$E_{ij} \in \Pic(Y)$ of classes of exceptional divisors for
a toric model. (Here for each $i$ the $E_{ij}$ are the exceptional divisors meeting the component $D_i$ of $D$.)
This is an ordered collection of disjoint interior
$(-1)$-curves. If $(Y,D)$ is not necessarily generic, then
by a {\it limiting configuration} in $\Pic(Y)$ we mean the
parallel transport (for the Gauss-Manin connection in a 
family of Looijenga pairs) of an exceptional configuration
on a generic pair.

We say that two exceptional configurations $\meij,\mfij$ for $(Y,D)$ have
the same {\it combinatorial type} if for each $i$ the number of divisors meeting $D_i$ is the same.

More generally, 
we extend the notion of exceptional or limiting 
configuration to mean
the data of a toric blowup $(Y',D') \to (Y,D)$ together with an
exceptional or limiting configuration on $(Y',D')$. 
\end{definition}

For generic pairs, limiting and exceptional
are the same, see Lemma~\ref{lim=exc_for_gen}.

\begin{definition} \label{excurves} 
A toric model 
$\pi: (Y,D) \to (\oY,\oD)$ is an iterated blowup
at some collection of (not necessarily distinct) 
points $q_{ij} \in \oD_i^o$ (where
$\oD_i^o \cong \bG_m$ is the complement of the nodes of $\oD$
along $\oD_i$). As such, the connected components of
the exceptional locus are disjoint unions 
of chains $E_1+\cdots+E_r$ of smooth rational curves with
self-intersections $-2,-2,\dots,-1$ (or just
a single $(-1)$-curve), where the length, $r$, is 
the number of times we blow up
at the corresponding point. This chain supports a unique
collection of $r$ reduced connected
chains, $C_1,\dots,C_r$, 
each of self-intersection $-1$, ordered by inclusion,
\[
C_1 = E_r, C_2 = E_r+E_{r-1}, \quad\dots\quad C_r = E_r+E_{r-1}+\cdots+E_1.
\]
Following
Looijenga, we refer to these chains as the \emph{exceptional curves}
for this toric model. Each such curve is determined by its
class, and they are partially ordered by
inclusion. Note if we produce a family 
$(\cY,\cD)/S$ of Looijenga pairs by varying the points $q_{ij}$ 
and choosing an 
order with which to make the iterated blowups, so that in the general
fibre we blow up distinct points, then each of these exceptional curves
on $Y$ is the limit of a unique smooth exceptional $(-1)$-curve on the
general fibre. 
\end{definition}

\begin{remark}\label{toric_iso_type}
Note that the isomorphism class of a toric
Looijenga pair $(\oY,\oD)$ 
is determined by the intersection numbers $\oD_i^2$.
Indeed, the isomorphism type of a smooth projective toric surface is determined by the self-intersection numbers of the components of the boundary divisor (because these determine the fan of the surface, see e.g. \cite{Fu93}, \S2.5).
\end{remark}

Note $(Y,D)$ together with
the classes $\meij$ of exceptional curves do not determine
by themselves the points $q_{ij} \in \oY$. Indeed, the classes determine
a birational contraction $p: (Y,D)\to (W,D)$, and $(W,D)$
is abstractly isomorphic to $(\oY,D)$, but further data is
needed to specify an identification: this is the data of a marking of
$D$. In the next couple of
lemmas we show that the positions of the $q_{ij}$ are determined
by the marked period point. From this the
global Torelli result contained in Theorem \ref{torelliI} will follow. 

\begin{lemma} \label{funny} Let $(Y,D)$ be a Looijenga pair.
For $\alpha \in \Aut^0(D)$ and 
$L \in \Pic(D)$ let 
\[
\psi_{\alpha}(L) = 
L^{-1} \otimes \alpha^*(L) \in \Pic^0(D)
\]
This gives a homomorphism $\psi: \Aut^0(D) \to \Hom(\Pic(Y),\Pic^0(D))$
via
\[
\psi(\alpha)(L)=\psi_{\alpha}(L|_D).
\]
Under the identifications $\Aut^0(D) = \bG_m^n$, $\Pic^0(D) = \bG_m$
of Lemma \ref{pic0D}, 
\[
\psi(\lambda_1,\dots,\lambda_n)(L)
= \prod_i \lambda_i^{\deg L|_{D_i}}
\]
for $L \in \Pic(D)$.
\end{lemma}
\begin{proof} 
It's enough to 
compute $\psi(1,\dots,1,\lambda,1,\dots,1)(\cO_D(q))$ for $q \in D_j^o$,
where $\lambda$ is in the $i^{th}$ place. Clearly this is
$\lambda^{\delta_{ij}}$, as required. \end{proof}

\begin{proposition} \label{basicseq}
There is a long exact sequence
\begin{align*}
1 \to \ker[\Aut(Y,D) \to \Aut(\Pic(Y))]&
\to \Aut^0(D)\\ \overset \psi \to 
&\Hom(\Pic(Y),\Pic^0(D))
\to \Hom(D^{\perp},\Pic^0(D)) \to 1
\end{align*}
where $\psi$ is the map of Lemma \ref{funny} and the
other maps are the canonical restrictions. 
\end{proposition}

\begin{proof} It is easy to see that if 
$(Y',D') \to (Y,D)$ is a toric blowup, then the result
for $(Y',D')$ implies the result for $(Y,D)$, so 
by Lemma \ref{tmexists} we can assume $(Y,D)$ has
a toric model $\pi:Y \to \oY$. 

We have the following commutative diagram of exact sequences:
\[
\xymatrix@C=30pt
{
&&0\ar[d]&0\ar[d]&&\\
&0\ar[r]\ar[d]&\Pic(\oY)\ar[r]\ar[d]^{\pi^*}&\bigoplus_i\ZZ\oD_i\ar[r]\ar[d]^=&N
\ar[r]\ar[d]&0\\
0\ar[r]&D^{\perp}\ar[r]\ar[d]&\Pic(Y)\ar[r]\ar[d]&
\bigoplus_i\ZZ D_i\ar[r]\ar[d]&N'\ar[r]\ar[d]&0\\
0\ar[r]&\bigoplus\ZZ E_{ij}\ar[r]^=&\bigoplus\ZZ E_{ij}\ar[r]\ar[d]&0\ar[r]
&0&\\
&&0&&&
}
\]
Here $N$ is dual to the character lattice, $M$, of the structure 
torus of $\oY$. The first row is the standard description of $A_1(\oY)$,
identified with $\Pic(\oY)$ by Poincar\'e duality, with the map from
$\Pic(\oY)$ given by $C\mapsto \sum_i (C\cdot \oD_i)\oD_i$. The map
to $N$ takes $\oD_i$ 
to the first lattice point $v_i$ along the ray of
the fan corresponding to $\oD_i$. This exact sequence is the dual of
the standard exact sequence describing $\Pic(\oY)$,
see e.g., \cite{Fu93}, \S 3.4.
The $E_{ij}$'s are the exceptional curves of $\pi$.
The map $\Pic(Y)\rightarrow\bigoplus_i \ZZ D_i$ is similarly given
by $C\mapsto \sum_i (C\cdot D_i) D_i$.

The kernel of $N\rightarrow N'$
is easily seen to be
the subgroup $S \subset N$ generated by the
rays in the fan for $\oY$ corresponding to boundary
divisors $\oD_i$ along which $\pi$ is not
an isomorphism. 

Note that
$N = \Hom(N,\bigwedge^2 N)$ via $n\mapsto (n'\mapsto n'\wedge n)$
and the orientation gives a trivialization
$\bigwedge^2 N = \bZ$, thus an identification $N = M$. Thus
$$
\Hom(N/S,\bG_m) \subset \Hom(N,\bG_m) =\Hom(M,\Gm)
$$
is the subgroup of homomorphisms to $\bG_m$ whose restriction to
$S$ is trivial. Equivalently, these are the automorphisms in 
$\Aut(\oY,\oD) = \Hom(M,\bG_m)$ fixing pointwise
those $\oD_i$ along which $\pi$ is not an isomorphism.
It's easy to see this is identified with 
\[
\ker\big(\Aut(Y,D) \to \Aut(\Pic(Y))\big).
\]
The result follows by applying $\Hom(\cdot,\Pic^0(D))$ to the
row of the above commutative diagram describing $\Pic(Y)$.
The fact that the middle map coincides with $\psi$ then follows 
from Lemma~\ref{funny}.
\end{proof}

We next show that for a toric Looijenga pair, any possible marked period point
can be realised by a particular choice of marking of $D$.

\begin{lemma} \label{toriccase} Let $(\oY,D=D_1 + \cdots +D_n)$ be a 
toric Looijenga pair, including an identification of the torus $T$ acting on $\oY$ with its open orbit. Let
$\bar\phi \in \Hom(\Pic(\oY),\Pic^0(D))$. Then there
are points $p_i \in D_i^o \subset \oY$ such that for any $L\in\Pic(\oY)$,
$$
\bar\phi(L) = (L|_D)^{-1} \otimes\bigotimes_{i=1}^n \cO_D((L \cdot D_i) p_i).
$$
Moreover, $T$ acts simply transitively on
the possible collections of $p_i$.
\end{lemma}

\begin{proof} Start with an arbitrary choice of $p_i \in D_i^o$.
The exact sequence of Proposition~\ref{basicseq}
reduces to 
\[
1 \mapright{} T \mapright{} \Aut^0(D) \mapright{\psi} 
\Hom(\Pic(\oY),\Pic^0(D)) \mapright{} 1.
\]
Denote the map $L\mapsto (L|_D)^{-1}\otimes\bigotimes_{i=1}^n\shO_D((L\cdot D_i)
p_i)$ by $\bar\phi'\in\Hom(\Pic(\oY),\Pic^0(D))$. Given any $\alpha
\in\Aut^0(D)$, using Lemma \ref{funny}, consider the map
\begin{align*}
L \mapsto {} & (L|_D)^{-1}\otimes\bigotimes_{i=1}^n\shO_D\big((L\cdot D_i)
\alpha^{-1}(p_i)\big)\\
= {} & \bar\phi'(L)\otimes \psi_{\alpha}\left(\bigotimes_{i=1}^n
\shO_D\big((L\cdot D_i)p_i\big)\right)\\
= {} & \bar\phi'(L)\otimes\psi(\alpha)(L).
\end{align*}
So this map coincides with $\bar\phi'\otimes\psi(\alpha)$.
Thus by replacing $p_i$ with $\alpha^{-1}(p_i)$ for some suitable choice of
$\alpha$, we obtain $\bar\phi=\bar\phi'$. Furthermore,
the possible choices of $p_i$ are a torsor for the kernel of $\psi$.
\end{proof}

\begin{lemma} \label{toriccasefamilies} Let $(\oY,D)$ be as in Lemma \ref{toriccase}. 
The structure of $\oY$ as a toric variety together with the orientation of $D$ gives a canonical identification
of $D_i^o$ with $\Gm$. Let $m_i\in D_i^o$ correspond to $-1\in\Gm$
under this identification. Define 
\[
p_i: \Aut^0(D) \to D_i^o, \quad \alpha \mapsto \alpha^{-1}(m_i).
\]
\begin{enumerate}
\item 
\[
\psi(\alpha)(L) = (L|_D)^{-1} \otimes \bigotimes_{i=1}^n \cO_D\big(
(L \cdot D_i) 
p_i(\alpha)\big) 
\in \Pic^0(D)
\]
for all $\alpha \in \Aut^0(D)$ ($\psi$ as in Lemma~\ref{funny}). 
\item Noting $\psi$ is surjective,
let $\gamma: \Hom(\Pic(\oY),\Pic^0(D)) \to \Aut^0(D)$ be a section of $\psi$. Let 
$\op_i: \Hom(\Pic(\oY),\Pic^0(D))  \to D_i^o$ be the composition $p_i \circ \gamma$. 
Then for each $\bar\phi\in \Hom(\Pic(\oY),\Pic^0(D))$, the points
$\bar p_i(\bar\phi)$ satisfy the conclusion of Lemma~\ref{toriccase} for
$\bar\phi$.
\end{enumerate}
\end{lemma}

\begin{proof} 
(1) amounts to showing that 
$$
L|_D = \bigotimes_{i=1}^n \cO_D((L \cdot D_i) m_i).
$$
It's enough to do this for an ample line bundle, so we can assume 
$(\oY,L)$ is the polarized toric surface given by a lattice polygon. In that
case take the section of $L$ given by a sum of monomials corresponding
to all lattice points on the boundary, with coefficients chosen so that
the restriction of the section to $D_i\cong\PP^1$ takes the form $(x+y)^{L\cdot
D_i}$. Its zero scheme is exactly $\sum (L \cdot D_i) m_i$.

For (2), note
\[
\bar\phi(L)=\psi(\gamma(\bar\phi))(L)=(L|_D)^{-1}\otimes
\bigotimes_{i=1}^n\shO_D\big((L\cdot D_i)
\bar p_i(\bar\phi)\big)
\]
by (1), as desired.
\end{proof}

The following contains most of the ideas needed for global Torelli,
showing that the marked period point
determines a marked Looijenga pair.

\begin{proposition} \label{closetotorelli} 
Let $(Y,D)$ be a Looijenga pair
and $\meij \subset \Pic(Y)$ the classes of exceptional
curves for a toric model of type $(\oY,D)$. 
Let $\phi \in \Hom(\Pic(Y),\Pic^0(D))$. 
\begin{enumerate}
\item
There is an inclusion $\Pic(\oY) \subset \Pic(Y)$
given by  pullback. Let $\bar\phi: \Pic(\oY) \to \Pic^0(D)$ be the restriction 
$\phi|_{\Pic(\oY)}$. Let $p_i \in D_i^o \subset \oY$ be given
by $\bar\phi$ from Lemma \ref{toriccase}. There are unique points
$q_{ij} \in D_i^o \subset \oY$ such that
$$
\phi(E_{ij}) = \cO_D(q_{ij})^{-1} \otimes \cO_D(p_i).
$$
Let $(Z,D)$ be the iterated blowup along the collection of points (possibly with
repetitions)
$q_{ij} \subset D_i^o \subset \oY$. There is a unique
isomorphism $\mu: \Pic(Y) \to \Pic(Z)$ preserving
boundary classes, and sending $E_{ij}$ to the
class of the corresponding exceptional curve. 
Under this identification, $\phi$ is the marked
period point of $((Z,D),p_i,\mu)$, as defined in \eqref{basiceq}.
\item
Suppose there
is a marking $r_i \in D_i^o \subset Y$ so that
$\phi$ is the marked period point for $((Y,D),r_i)$. 
Then $\mu$ is induced by
a unique isomorphism of Looijenga pairs between $(Y,D)$ and $(Z,D)$ 
which sends $r_i$ to $p_i$. 
\end{enumerate}
\end{proposition}

\begin{proof} 
(1) is immediate from the
construction. So we assume we have the marking $r_i \in D_i^o$
as in (2). By assumption there is a birational
map $\pi: Y \to \oY$ with exceptional curves $\meij$, and 
$\pi^*: \Pic(\oY) \to \Pic(Y)$ is the inclusion of
(1). Now by definition of the marked period point,
the points $\pi(r_i)$ satisfy the conclusions of Lemma \ref{toriccase}
for $\bar\phi$. Thus by
the uniqueness statement in that lemma, we can change
$\pi$ (composing by a translation in the structure torus of $\oY$) 
and assume $\pi(r_i) = p_i$. The points 
$\pi(E_{ij} \cap D_i)$ satisfy the conditions on the $q_{ij}$, 
so by uniqueness $\pi(E_{ij} \cap D_i) = q_{ij}$. 
Thus $\pi$ is exactly
the same iterated blowup as $Z$, and so clearly 
$(Y,D)$ and $(Z,D)$, together with the markings of
their boundaries, are isomorphic, by an isomorphism
inducing $\mu$. This isomorphism is unique by 
Proposition \ref{basicseq}.
\end{proof} 

\begin{corollary} \label{torcor} Let $(Y,D),(Y',D)$ be Looijenga
pairs (resp.\ pairs with marked boundary), 
having toric models of the same combinatorial type.
Let $\phi,\phi'$ be the period points (resp.\ the marked
period points).
Then there is a unique isomorphism of lattices
$\mu: \Pic(Y) \to \Pic(Y')$ preserving the boundary classes
and the exceptional curves for the toric models. 
The isomorphism $\mu$ is induced by an isomorphism $f$ of Looijenga
pairs (resp. pairs with marked boundary) 
iff $\phi' \circ \mu = \phi$, and in that case the
possible $f$ form a torsor for 
\[
\ker\big(\Aut(Y,D) \to \Aut(\Pic(Y))\big)
\]
(resp.\ $f$ is unique). 
\end{corollary}

\begin{proof} The marked case is immediate from 
Proposition \ref{closetotorelli}. 
For the unmarked case, write $\bar\phi,\bar\phi'$ for
the period points defined by \eqref{unmarkedperiod}, and assume $\bar\phi'\circ\mu = \bar\phi$. 
Choose arbitrary markings of the boundaries
of $Y,Y'$, with marked period points $\phi,\phi'$. Now by
Proposition \ref{basicseq} we can adjust the marking of the boundary
of $Y$ so $\phi' \circ \mu = \phi$. The final torsor statement is
clear from Proposition \ref{basicseq}.\end{proof}

For a Looijenga pair $(Y,D)$, we define the monodromy group as follows.
For $(\cY,\cD)/S$ an analytic family of Looijenga pairs over a connected base $S$, a base point $s \in S$, an identification $(\cY_s,\cD_s) = (Y,D)$, and a path $\gamma \colon [0,1] \rightarrow S$ with $\gamma(0)=\gamma(1)=s$,
we obtain a \emph{monodromy transformation} $\rho(\gamma) \in \Aut(\Pic(Y))$ by parallel transport along the loop $\gamma$. The \emph{monodromy group} of $(Y,D)$ is the subgroup of $\Aut(\Pic(Y))$ consisting of all monodromy transformations.

\begin{remark}
We show in Theorem~\ref{strong_monodromy} that the full monodromy group is realized by an analytic family over a smooth base.
\end{remark}

\begin{lemma}\label{PL}
Let $(Y,D)$ be a Looijenga pair, $\Phi$ the associated set of roots, and $W$ the Weyl group of $\Phi$. Then $W$ is contained in the monodromy group of $(Y,D)$.
\end{lemma}
\begin{proof}
Given $\alpha \in \Phi$, by definition there exists a family of Looijenga pairs $(\cY,\cD)/S$, a path $\gamma \colon [0,1] \rightarrow S$, and an identification $(Y,D) = (\cY_{\gamma(0)},\cD_{\gamma(0)})$, 
such that the parallel transport of the class $\alpha \in \Pic(Y)=H^2(Y,\bZ)$ is realized by an internal $(-2)$-curve $C$ on $(Y',D'):=(\cY_{\gamma(1)},\cD_{\gamma(1)})$. 
Let $(\oY',\oD')$ denote the contraction of $C$. Let $(\cY',\cD')/(0 \in T)$ and $(\ocY',\ocD')/(0 \in \oT)$ denote the versal deformations of $(Y',D')$ and $(\oY',\oD')$ respectively. 
Then $(0 \in T)$ and $(0 \in \oT)$ are smooth germs, the locus $H \subset \oT$ of singular fibers is a smooth hypersurface, 
and there is a finite morphism $T \rightarrow \oT$ of degree $2$ with branch locus $H$ and a birational proper morphism $\cY' \rightarrow \ocY' \times_{\oT} T$ which restricts to the minimal resolution of each fiber. See \cite{L81}, II.2.4.
The monodromy of the family around $H$ is given by the Picard-Lefschetz reflection in the class of $[C]$.
Now, using the path $\gamma$, we deduce that the reflection $s_{\alpha}$ lies in the monodromy group of $(Y,D)$.
\end{proof}

\begin{lemma} 
\label{C++lemma}
Let $(Y,D)$ be a Looijenga pair. Let $E \in \Pic(Y)$ be
a class with $E^2 = K_Y \cdot E = -1$. The following are equivalent:
\begin{enumerate}
\item $E \cdot H > 0$ for some nef divisor $H$.
\item $E$ is effective.
\end{enumerate}
The cones $C^{++}$ and $C^{++}_D$ defined in Definition
\ref{cpdef} are invariant under parallel transport
for deformations of Looijenga pairs, and under the action of $W_Y$. 
\end{lemma}  

\begin{proof}
Obviously (2) implies (1). Riemann--Roch gives (1) implies (2).

Given a family of Looijenga pairs over a base scheme $S$, working locally analytically on $S$ we can choose an ample divisor, $H$, on the total 
space and then compute $C^{++}$ on each fibre using the restriction of $H$. 
From this deformation invariance is clear. Invariance under $W_Y$ follows from Lemma~\ref{PL}.
\end{proof}

\begin{lemma} \label{Moricones} Let $(Y,D)$ be a Looijenga pair. Let 
$\cM \subset \Pic(Y)$ denote the set of classes of $(-1)$-curves not contained in $D$.
\begin{enumerate}
\item
Let $C \subset Y$ be an irreducible curve. Either $C^2 \ge 0$ or $[C] \in \Pic(Y)$
is in the union of $\cM$, $\Delta_Y$ and $\{[D_i]\,|\,1\le i\le n\}$. 
\item
Let $H \in \Pic(Y)$ be an ample class. Then the closure of the
Mori cone of curves 
$\overline{\NE}(Y)$ is the closure of the convex hull of the union of
$$
C^{+} :=\{x \in \Pic(Y)\otimes_{\ZZ}\RR \ | \ x^2 > 0, x \cdot H > 0\}
$$
together with $\Delta_Y$, $\cM$ and $\{[D_i]\,|\,1\le i\le n\}$. 
Equivalently, by Lemma~\ref{C++lemma}, $\overline{\NE}(Y)$ is the closure of the convex hull of the union of $C^+$, $\Delta_Y$, $\tilde{\cM}$, and $\{[D_i]\,|\,1\le i\le n\}$,  where 
$$\tilde{\cM}=\{ E \in \Pic(Y) \ | \ E^2=K_Y \cdot E = -1 \mbox{ and } E \cdot H > 0 \}$$
for some ample divisor $H$ (as in the definition of $C^{++}$).
\end{enumerate}
\end{lemma}

\begin{proof} 
For (1), let $C \subset Y$, $C \not \subset D$, 
be irreducible. If $C^2 <0$ then
$C \in \Delta_Y \cup \cM$ by adjunction. 

For (2), note $C^{+} \subset \NE(Y)$ by Riemann-Roch and if $C$ is 
effective with
$C^2\ge 0$, then $C$ is contained in the closure of $C^+$. 
The description of the Mori cone then follows from (1).
\end{proof}

\begin{lemma} \label{Nefcones} Let $(Y,D)$ be a Looijenga pair
and $H \in \Pic(Y)$ an ample class. 
Then $\Nef(Y) \subset H^2(Y,\bR)$ 
is the closure of the subcone of 
$C^{++}_D$ defined by the inequalities
$x \cdot \alpha \ge 0$ for all $\alpha \in \Delta_Y$.
\end{lemma}

\begin{proof} 
Since $\Nef(Y)$ is the dual cone to $\NE(Y)$, this follows immediately
from Lemma~\ref{Moricones}, (2).
\end{proof}

\begin{proof}[Proof of the global Torelli, Theorem \ref{torelliI}]
If $\mu = f^*$ for an isomorphism $f$ then
$\mu$ obviously
satisfies the conditions, and the possibilities for
$f$ are a torsor for $\ker(\Aut(Y,D)\rightarrow\Aut(\Pic(Y)))$,  as in Corollary~\ref{torcor}. 
This is identified in the proof of Proposition~\ref{basicseq}
with $\Hom(N',\Gm)$.

Now assuming we have such a $\mu$, we show
it is induced by an isomorphism of pairs.
We can replace $Y_1$ by a toric blowup 
and $Y_2$ by the corresponding toric blowup, and 
so by Lemma \ref{tmexists} we can assume $Y_1$ has a toric model. Then
$\mu(\Nef(Y_1)) = \Nef(Y_2)$
by Lemma \ref{Nefcones}. Thus the same is true of the Mori cones of curves by
duality. Note also that $\mu(K_{Y_1}) = K_{Y_2}$ since $D$ is anti-canonical. 

The exceptional locus of a toric model $Y_1 \rightarrow \oY_1$ is a disjoint union of chains of interior smooth rational curves $F_1,\ldots,F_r$
with self-intersection numbers $-2,-2,\ldots,-2,-1$, such that $F_j$ is disjoint from $D$ for $j<r$ and $F_r$ meets $D$ transversely in one point.
(Such a chain is the exceptional locus over a point $p \in D$ which is blown up $r$ times.)
By assumption $\mu(\Delta_{Y_1})=\Delta_{Y_2}$, so $\mu$ sends internal $(-2)$-curves to internal $(-2)$-curves.
Also, the class $x$ of a $(-1)$-curve is characterized by $x^2=-1$, $x \cdot K =-1$, and $x$ generates an extremal ray of the Mori cone. 
Thus $\mu$ sends interior $(-1)$-curves to interior $(-1)$-curves. 
Also, since $\mu$ preserves the intersection product, the curves in $Y_2$ corresponding to the exceptional locus of $Y_1 \rightarrow \oY_1$ intersect in the same way, that is, they form a disjoint union of chains. 
Hence there is a birational morphism $(Y_2,D) \rightarrow (\oY_2,\oD)$ which contracts these curves, and is given by a sequence of blowups of the same combinatorial type as $(Y_1,D) \rightarrow (\oY_1,\oD)$.

We claim that the surface $(\oY_2,\oD)$ is toric. Let $(Y,D)$ be a Looijenga pair, and write $e(X)=\sum (-1)^i\dim H^i(X,\bR)$ for the Euler number of a topological space $X$. 
If $(Y',D') \rightarrow (Y,D)$ is a toric blowup then $Y' \setminus D' = Y \setminus D$ so in particular $e(Y'\setminus D') = e(Y\setminus D)$. 
If $(Y',D') \rightarrow (Y,D)$ is a birational morphism of Looijenga pairs given by blowing up a smooth point of $D$ (and defining $D'$ to be the strict transform of $D$) then $e(Y'\setminus D')=e(Y\setminus D)+1$. 
If $(Y,D)$ is toric then $e(Y \setminus D)=e((\bC^{\times})^2)=0$.
Now it follows from the existence of toric models (Lemma~\ref{tmexists}) that a Looijenga pair $(Y,D)$ satisfies $e(Y\setminus D) \ge 0$ with equality iff $(Y,D)$ is toric. 
In our situation we have $e(\oY_1 \setminus \oD) = e(\oY_2 \setminus \oD)$ 
(because $e(Y_1 \setminus D) = e(Y_2 \setminus D)$ and the toric models $(Y_1,D) \rightarrow (\oY_1,\oD)$ and $(Y_2,\oD) \rightarrow (\oY_2,D)$ have the same number of exceptional curves). 
Thus $(\oY_1,\oD)$ toric implies $(\oY_2,\oD)$ toric.

Next observe that the toric pairs $(\oY_1,\oD)$ and $(\oY_2,\oD)$ are isomorphic.  
Indeed, the self-intersection numbers $\oD_i^2$ for $\oY_1$ and $\oY_2$ coincide because the self-intersection numbers $D_i^2$ for $Y_1$ and $Y_2$ coincide 
and the toric models $(Y_1,D) \rightarrow (\oY_1,\oD)$ and $(Y_2,D) \rightarrow (\oY_2,\oD)$ have the same combinatorial type. 
So $(\oY_1,\oD)$ and $(\oY_2,\oD)$ are isomorphic by Remark~\ref{toric_iso_type}.

Now we may apply Corollary \ref{torcor}. 
\end{proof}

\section{The weak Torelli theorem}\label{weaktorellisection}

The following result is due to R.~Friedman.

\begin{theorem}\label{Friedman_roots}(\cite{F13}, Theorem~2.14.)
The set $\Phi$ of roots coincides with the set of classes $\alpha \in \Pic(Y)$ such that $\alpha^2=-2$, $\alpha \cdot D_i =0$ for each $i$, and the associated hyperplane $\alpha^{\perp}$ meets the interior of $C^{++}_D$.
\end{theorem}

We recall the following statement about the action of Weyl groups:

\begin{theorem} \label{thWeylchambers} The arrangement
of hyperplanes
$$
\alpha^{\perp} \subset C^{++},\quad \alpha \in W_Y \cdot \Delta_Y
$$
is locally finite. 
The group $W_Y$ acts simply transitively on the Weyl
chambers, and each chamber is
a fundamental domain for the action of $W_Y$ on $C^{++}$.  One chamber 
is defined by the inequalities $x \cdot \alpha \geq 0$ for all $\alpha \in \Delta_Y$
(and for each $\alpha \in \Delta_Y$ the equation $x \cdot \alpha = 0$ defines a codimension one face of this chamber).
The analogous statements hold for the Weyl chambers of $C^{++}_{D}$.
\end{theorem}

\begin{proof} The analogous statement for chambers in $C^{+}$ is 
a basic result in the theory of hyperbolic reflection groups,
see \cite{D08}, Theorem 2.1. This immediately implies the result
for the chambers in $C^{++}$ or $C^{++}_D$ , as these full dimensional
subcones of $C^+$ are preserved by $W_Y$, see Lemma \ref{C++lemma}.
The closure of the chamber in $C^{++}_D$ defined by $x \cdot \alpha \ge 0$ for each $\alpha \in \Delta_Y$ is identified with the nef cone of $Y$ by Lemma~\ref{Nefcones}.
By definition the elements of $\Delta_Y$ are the classes of $(-2)$-curves on $Y$ and thus define codimension one faces of the nef cone.
\end{proof}

\begin{lemma} \label{-2effective}
Let $(Y,D)$ be a Looijenga pair.
Let $L$ be a line bundle on $Y$ such that $L^2=-2$ and $L|_D \simeq \cO_D$.
Then $h^0(L)>0$ or $h^0(L^{-1})>0$.
\end{lemma}
\begin{proof}
Suppose $H^0(L) = 0$.
Using the exact sequence
$$0 \rightarrow L \otimes \cO_Y(-D) 
\rightarrow L \rightarrow \cO_D \rightarrow 0$$
we see that $H^0(L \otimes \cO_Y(-D))=0$ and 
$H^1(L \otimes \cO_Y(-D)) \neq 0$.
Equivalently, by Serre duality, $H^1(L^{-1}) \neq 0$ and $H^2(L^{-1})=0$.
Now by the Riemann--Roch formula
$$h^0(L^{-1}) > \chi(L^{-1})=\chi(\cO_Y)+\frac{1}{2}L^{-1}\cdot(L^{-1}-K_Y)=0.$$
\end{proof}

\begin{proposition} \label{simplerootsgenerate}
Let $(Y,D)$ be a Looijenga pair. Then $\Phi_Y=W_Y \cdot \Delta_Y$.
\end{proposition}
\begin{proof}
(cf. \cite{F13}, Proof of Theorem~2.14). 
Note that $W$ preserves $\Phi$ by Lemma~\ref{PL} and $W_Y$ preserves the period point $\phi_Y \colon D^{\perp} \rightarrow \bG_m$.
It follows that $W_Y \cdot \Delta_Y \subset \Phi_Y$. 
Conversely, given $\alpha \in \Phi_Y$, we show $\alpha \in W_Y \cdot \Delta_Y$. 
By Theorem~\ref{Friedman_roots} there exists a class $x$ in the interior of $C^{++}_D$ such that $x \cdot \alpha = 0$. 
In particular $x \cdot [D_i] > 0$ for each $i$.
We may assume $x$ is an integral class, say $x=[H]$. By Lemma~\ref{Nefcones} and Theorem~\ref{thWeylchambers}, replacing $x$ and $\alpha$ by $wx$ and $w\alpha$ for suitable $w \in W_Y$, we may assume $x$ lies in the nef cone of $Y$. Also $x^2>0$ (because $\alpha^2=-2<0$ and $x \cdot \alpha = 0$). So $H$ is nef and big. 
By Lemma~\ref{-2effective}, replacing $\alpha$ by $-\alpha$ if necessary, we may assume that $\alpha$ is effective, say $\alpha = \sum a_i[C_i]$ for some irreducible curves $C_i \subset Y$ and $a_i \in \bN$.
Now $\alpha \cdot H = 0$ implies $C_i \cdot H = 0$ for each $i$. In particular no $C_i$ is a component of $D$, so $\alpha \cdot D = 0$ implies $C_i \cdot D =0$ for all $i$. Also, the span of the classes of the $C_i$ is negative definite. 
Now by adjunction each $C_i$ is a $(-2)$-curve, and $\bigcup C_i$ is a configuration of $(-2)$-curves with dual graph a Dynkin diagram of type $A$, $D$, or $E$. (Note that $\bigcup C_i$ is connected because it is the support of the cycle $\sum a_iC_i$ with square $-2$.)
Finally, the Weyl group of a root system of type $A$, $D$, or $E$ acts transitively on the set of roots (and the roots are precisely the elements $\beta$ of the root lattice such that $\beta^2=-2$). So $\alpha \in W_Y \cdot \Delta_Y$.
\end{proof}

\begin{corollary} \label{genericlocus}
Let $(Y,D)$ be a Looijenga pair. Then $(Y,D)$ is generic iff $\phi_Y(\alpha) \neq 1$ for all $\alpha \in \Phi$.
\end{corollary}
\begin{proof}
By definition $(Y,D)$ is generic iff $\Delta_Y = \emptyset$. 
This is equivalent to $\Phi_Y = \emptyset$ by Proposition~\ref{simplerootsgenerate}.
\end{proof}

\begin{proof}[Proof of the weak Torelli Theorem]
Note that $W_{Y_1}$ fixes $\phi_{Y_1}$ and the $[D_i]$ by the definitions, and preserves $C^{++}$ by Lemma~\ref{C++lemma}. So the conditions on the isomorphism $\mu$ of lattices are necessary.
Conversely, suppose given $\mu$ satisfying the hypotheses.
The isomorphism $\mu$ satisfies $\mu(\Phi)=\Phi$ by Theorem~\ref{Friedman_roots} and hence
$\mu(\Phi_{Y_1})=\Phi_{Y_2}$ by condition (4) of the statement of weak Torelli.
Also $\Phi_{Y_i} = W_{Y_i} \cdot \Delta_{Y_i}$ for each $i=1,2$ by Proposition~\ref{simplerootsgenerate}.
Thus $\mu$ sends the $W_{Y_1}$-Weyl chambers of $C^{++}_D \subset \Pic(Y_1)_{\bR}$ to the $W_{Y_2}$-Weyl chambers of $C^{++}_D \subset \Pic(Y_2)_{\bR}$. 
Since $W_{Y_1}$ acts simply transitively on the $W_{Y_1}$-Weyl chambers of $C^{++}_D$, 
there exists a unique $g \in W_{Y_1}$ such that $\mu \circ g$ satisfies $\mu(\Delta_{Y_1})=\Delta_{Y_2}$. Now the global Torelli Theorem applies. 
\end{proof}

\section{First properties of the monodromy group}

\begin{proposition} \label{generics}
Let $(Y,D)$ be a Looijenga pair. Let $(0 \in \Def(Y,D))$ denote the versal deformation space of the pair and $T'_Y=\Hom(D^{\perp},\bG_m)$. 
\begin{enumerate}
 \item The local period mapping $$\phi \colon (0 \in \Def(Y,D)) \rightarrow (\phi_Y \in T'_Y)$$ 
is a local analytic isomorphism.
 \item The locus of generic pairs in $\Def(Y,D)$ is the complement of the inverse image under $\phi$ of the countable union of hypertori 
$$T'_{\alpha}=\{ \psi \in T'_Y \ | \ \psi(\alpha) = 1 \}$$
for $\alpha \in \Phi$.
\end{enumerate}
In particular, every Looijenga pair is a deformation of a generic pair.
\end{proposition}
\begin{proof} 
The period mapping is a local isomorphism by \cite{L81}, II.2.5.

Statement (2) follows from Corollary~\ref{genericlocus}.
\end{proof}

\begin{definition} 
Let $(Y,D)$ be a Looijenga pair. Let $\Adm_Y$ denote the subgroup of automorphisms of the lattice $\Pic(Y)$ preserving the boundary classes $[D_i]$ and the cone $C^{++}$ (see Definition~\ref{cpdef}).
We say an automorphism $\theta$ of $\Pic(Y)$ is \emph{admissible} if $\theta \in \Adm_Y$.
\end{definition}

\begin{lemma} \label{MonAdmPhi}
Let $(Y,D)$ be a Looijenga pair. The group $\Adm_Y$ contains the monodromy group of $(Y,D)$ and preserves $\Phi$.
\end{lemma}

\begin{remark}
In fact we show in Theorem~\ref{strong_monodromy} that $\Adm_Y$ is equal to the monodromy group.
\end{remark}

\begin{proof}
The monodromy group preserves the cone $C^{++}$ by Lemma~\ref{C++lemma}, so it is contained in $\Adm_Y$. 
The group $\Adm_Y$ preserves $\Phi$ by Theorem~\ref{Friedman_roots}.

\end{proof}

\begin{lemma}\label{AdmLemma}
Let $(Y,D)$ be a generic Looijenga pair and $\theta \colon \Pic(Y) \rightarrow \Pic(Y)$ an isomorphism of lattices such that $\theta([D_i])=[D_i]$ for each $i$. The following conditions are equivalent:
\begin{enumerate}
 \item $\theta \in \Adm_Y$.
 \item $\theta(\Nef(Y))=\Nef(Y)$.
 \item There exists $H \in \Pic(Y)$ such that $H$ and $\theta(H)$ are ample.
\end{enumerate}
\end{lemma}
\begin{proof}
By definition $\theta \in \Adm_Y$ iff $\theta(C^{++})=C^{++}$, and $\Nef(Y)=\overline{C^{++}_D}$ by Lemma~\ref{Nefcones} because $(Y,D)$ is generic. So (1) implies (2).
Clearly (2) implies (3) (because the ample cone is the interior of the nef cone). Finally, suppose $\theta$ satisfies (3). Then $\theta$ preserves the set $\tilde{\cM}$ and hence the cone $C^{++}$ (see Definition~\ref{cpdef}). So (3) implies (1) and the 
equivalence of the statements is proved.
\end{proof}

\begin{lemma} \label{lim=exc_for_gen}
Let $(Y,D)$ be a generic Looijenga pair. Then any limiting configuration on $Y$ is an exceptional configuration.
\end{lemma}
\begin{proof}
By definition, a limiting configuration on $Y$ is the parallel transport of an exceptional configuration on a generic pair $(Y_0,D)$.
Note that $\Nef(Y_0)$ and $\Nef(Y)$ are identified under parallel transport 
(because for a generic pair the nef cone coincides with $\overline{C^{++}_D}$ by Lemma~\ref{Nefcones}, and this cone is invariant under parallel transport by Lemma~\ref{C++lemma}). 
The elements of the exceptional configuration on $Y_0$ define codimension one faces of $\Nef(Y_0)$. 
Hence the elements $E_{ij}$ of the limiting configuration define codimension one faces of $\Nef(Y)$. 
Now by Lemma~\ref{Moricones}(1) and the intersection numbers it follows that the $E_{ij}$ are a collection of disjoint interior $(-1)$-curves. 
As in the proof of the global Torelli theorem, contracting these curves yields a toric pair $(\oY,\oD)$, so $\{E_{ij}\}$ is an exceptional configuration.
\end{proof}

\begin{theorem}\label{cremona}
Let $(Y,D)$ be a Looijenga pair.
The group $\Adm_Y$ acts simply transitively on the set of limiting configurations of (any given) combinatorial type.
\end{theorem}
\begin{proof}
We may assume that $(Y,D)$ is generic by Proposition~\ref{generics}.

We show that if $\pi \colon (Y',D') \rightarrow (Y,D)$ is a toric blowup, then we have a natural identification $\Adm_{Y'}=\Adm_Y$. 
Note that $(Y,D)$ generic implies $(Y',D')$ generic by the definition of generic, so we may use the equivalent conditions above.
We may assume that $\pi$ is a simple toric blowup, with unique exceptional divisor $E$.
Given $\theta \in \Adm_Y$, we define a homomorphism $\theta' \colon \Pic(Y') \rightarrow \Pic(Y')$ by $\theta'(\pi^*\alpha)=\pi^*\theta(\alpha)$ and $\theta'([E])=[E]$. We claim that $\theta' \in \Adm_{Y'}$. It is clear that $\theta'$ is an isomorphism of lattices and 
$\theta'([D_i'])=[D_i']$ for each component $D_i'$ of the boundary $D' \subset Y'$. Letting $H \in \Pic(Y)$ be ample, then $\theta(H)$ is also ample on $Y$. Now for $N \in \bN$ sufficiently large, $H':=N\pi^*H - E$ and 
$\theta'(H')$ are ample on $Y'$. So $\theta' \in \Adm_{Y'}$. The map $\Adm_Y \rightarrow \Adm_{Y'}$ defined in this way is clearly a group homomorphism. Conversely, given $\theta' \in \Adm_{Y'}$, we have $\theta'([E])=[E]$. 
Thus we can define $\theta \colon \Pic(Y) \rightarrow \Pic(Y)$ by restricting $\theta'$ to $E^{\perp}$ and using the identification $E^{\perp}=\Pic(Y)$ given by $\pi^*$. 
Then $\theta$ is an isomorphism of lattices and $\theta([D_i])=[D_i]$ for each $i$.
Now letting $H'$ be ample on $Y'$, then $\theta'(H')$ is also ample on $Y'$. Hence $H:=\pi_*H'$ and $\theta(H)=\pi_*(\theta'(H'))$ are ample on $Y$, so $\theta \in \Adm_Y$. This defines a homomorphism $\Adm_{Y'} \rightarrow \Adm_{Y}$ which is 
clearly the inverse of the homomorphism described above. 
 
Let $\theta \in \Adm_Y$, and let $\{E_{ij}\}$ be an exceptional configuration on a toric blowup $(Y',D')$ of $(Y,D)$. We show that $\{\theta(E_{ij})\}$ is another exceptional configuration of the same combinatorial type.
Using the identification $\Adm_Y=\Adm_{Y'}$ proved above, we may assume $Y=Y'$. We have $\theta(\Nef(Y))=\Nef(Y)$, so the $\theta(E_{ij})$ define codimension one faces of $\Nef(Y)$. We can now conclude as in the proof of Lemma~\ref{lim=exc_for_gen} above.

Conversely, let $\{E_{ij}\}$, $\{F_{ij}\}$ be two exceptional configurations on $(Y,D)$ of the same combinatorial type. Clearly there is a unique isomorphism of lattices $\theta \colon \Pic(Y) \rightarrow \Pic(Y)$ such that $\theta([D_i])=[D_i]$ for all $i$ and
$\theta([E_{ij}])=[F_{ij}]$ for all $i$ and $j$. We must show that $\theta \in \Adm_Y$. 
Let $\pi \colon (Y,D) \rightarrow (\oY,\oD)$ denote the contraction of the $\{E_{ij}\}$, and $\pi' \colon (Y,D) \rightarrow (\oY,\oD)$ the contraction of the $\{F_{ij}\}$. (Note that the toric pairs $(\oY,\oD)$ obtained by the contractions are 
(non-canonically) isomorphic because the exceptional configurations have the same combinatorial type.) Let $\oH=\sum a_i\oD_i$ be ample on $\oY$. 
Then for $N \in \bN$ sufficiently large both $H=N\pi^*\oH - \sum E_{ij}$ and $\theta(H)=N(\pi')^*\oH - \sum F_{ij}$ are ample on $Y$. So $\theta \in \Adm_Y$.
\end{proof}

\section{Automorphisms, universal families, and the monodromy group} \label{univfam}

Given $(Y_0,D)$ a generic Looijenga pair, let $(Y_e,D)$ be a Looijenga pair deformation equivalent to $(Y_0,D)$ with period point $\phi_{Y_e}$ given by $\phi_{Y_e}(\alpha)=1$ for all $\alpha \in D^{\perp} \subset \Pic(Y_e)$.
(Note that existence of $(Y_e,D)$ follows from the construction of Proposition~\ref{closetotorelli}, and $(Y_e,D)$ is uniquely determined up to isomorphism by the weak Torelli theorem.)

We analyze the relationship between the Weyl group $W$, the group $\Adm_{Y_0}$, and the automorphisms groups of Looijenga pairs deformation equivalent to $(Y_0,D)$.

\begin{theorem}
\label{exactseqtheorem}
Let $(Y_0,D)$ be a generic Looijenga pair and define $(Y_e,D)$ as above. 
Then $W \subset \Adm_{Y_0}$
is a normal subgroup and there is an exact sequence 
\begin{equation}\label{autoexactseqv1}
1 \to \Hom(N',\bG_m) \to \Aut(Y_e,D) \to \Adm_{Y_0}/W \to 1
\end{equation}
where $N'$ is the group defined in Theorem~\ref{torelliI}.

More generally, for $(Y,D)$ an arbitrary Looijenga pair deformation equivalent to $(Y_0,D)$, let 
$\Hodge_Y \subset \Adm_{Y_0}$ denote the stabilizer of the period point $\phi_Y$ (for some choice of marking of $\Pic(Y)$).
Then we have an exact sequence
\begin{equation}\label{autoexactseqv2}
1 \rightarrow \Hom(N',\bG_m) \rightarrow \Aut(Y,D) \rightarrow \Hodge_Y/W_Y \rightarrow 1
\end{equation}
\end{theorem}

\begin{proof}
Note $W_Y \subset \Adm_{Y_0}$ and $\Adm_{Y_0}$ preserves $\Phi$ by Lemma~\ref{PL} and Lemma~\ref{MonAdmPhi}. 
Now, since 
$$\Phi_Y=\{ \alpha \in \Phi \ | \ \phi_Y(\alpha)=1 \},$$
the group $\Hodge_Y$ preserves $\Phi_Y$ and $W_Y \subset \Hodge_{Y}$ is normal.
The image of $\Aut(Y,D)$ in $\Aut(\Pic(Y_0))$ has trivial intersection with $W_{Y}$,
since it preserves the Weyl chamber $\Nef(Y) \subset \overline{C^{++}_D}$, while the Weyl group
acts simply transitively on the chambers. Take $g \in \Adm_{Y_0}$. Composing
$g$ with an element of $W_Y$ we can assume $g$ preserves the Weyl chamber 
$\Nef(Y)$, and thus $\Delta_{Y}$ (as each $\alpha \in \Delta_Y$ corresponds to a codimension one face of
the chamber). Now $g$ is in the image of $\Aut(Y,D)$ iff it fixes the period point $\phi_Y$
by the global Torelli Theorem. Thus the homomorphism $\Aut(Y,D) \to \Hodge_{Y}/W_Y$ 
is surjective. Now the exactness follows from Proposition~\ref{basicseq}.

Finally, for $Y_e$ the period point equals the identity element of $\Hom(D^{\perp},\bG_m)$, so $\Hodge_{Y_e}=\Adm_{Y_0}$ and $W_{Y_e}=W$ by Proposition~\ref{simplerootsgenerate}.
\end{proof}

\begin{remark}
Note that the description of the automorphism groups of certain Looijenga pairs in \cite{L81}, Corollary~I.5.4 is incorrect as stated.
(The assumption that the automorphism acts trivially on $D^{\perp}$ should be added to the statement. Moreover, the group $\bZ/s\bZ \times \bZ/2\bZ$ in the statement should be replaced by the dihedral group of order $2s$.)
The group $N'$ is trivial in the cases studied by Looijenga so $\Aut(Y,D)=\Hodge_Y/W_Y$. 
\end{remark}

\begin{example} \label{nonfgcone}
We give an example where $\Adm_{Y_0}/W$ is nontrivial (in fact, infinite). 
Let $D$ be a cycle of seven $(-2)$-curves.
Then one can show that 
$\Aut(Y_e,D)$ is infinite. 
(Indeed, since $\cO_{Y_e}(D)|_D \simeq \cO_D$, there is an elliptic fibration $f \colon Y_e \rightarrow \bP^1$ with $f^{-1}(\infty)=D$. 
Moreover, the fibration $f$ is relatively minimal because $K_X=-D$.
So there is an action of the Mordell--Weil group $\MW(f)$ of sections of $f$ on $(Y_e,D)$ given by translation by the section on the smooth fibers of $f$.
Finally, $\MW(f)$ is infinite by \cite{MP86}, Theorem~4.1 (or a short root theoretic calculation, cf. Example~\ref{W_trivial_Adm_infinite} below).)
The group $N'$ is finite because 
$[D_1],\ldots,[D_n] \in \Pic(Y_0)$ are linearly independent. 
Hence $\Adm_{Y_0}/W$ is infinite by Theorem~\ref{exactseqtheorem}. 
\end{example}

We note by way of comparison:

\begin{lemma} \label{looadm} In the cases Looijenga considers in \cite{L81} we have $\Adm_Y=W$.
\end{lemma}
\begin{proof}  We use \cite{L81}, Proposition~I.4.7, p.~284.
By definition $\operatorname{Cr}(Y,D)$ is the group of automorphisms of the lattice $\Pic(Y)$ 
preserving the ample cone of $Y$ and the boundary divisors $D_1,\ldots,D_n$. 
We may assume $(Y,D)$ is generic, that is, in Looijenga's notation $B^n = \emptyset$. 
Then $\Adm_Y=\operatorname{Cr}(Y,D) = W$ by Lemma~\ref{AdmLemma} and \cite{L81}, I.4.7.
\end{proof}

\begin{example}\label{Delta_infinite}
We describe an example of a Looijenga pair $(Y'_e,D'_e)$ such that the set $\Delta$ of internal $(-2)$-curves on $Y'_e$ is infinite.

Let $(Y_e,D_e)$ be the Looijenga pair of Example~\ref{nonfgcone}. Then there is an elliptic fibration $f \colon Y_e \rightarrow \bP^1$ with $D=f^{-1}(\infty)$ and such that $f$ has infinitely many sections. 
Each section $C \subset Y_e$ is a $(-1)$-curve (because $-K_{Y_e} \cdot C = D \cdot C = 1$). 
Any two sections meeting the same component of $D$ intersect $D$ at the same point (because $\phi_{Y_e}$ is trivial by definition). 
So there is a point $p$ in the smooth locus of $D_e$ such that there are infinitely many $(-1)$-curves on $Y_e$ passing through $p$. 
Now let $(Y',D')$ denote the blowup of $p \in Y_e$ together with the strict transform of $D_e$. Then clearly $\phi_{Y'}$ is also trivial, so $(Y',D')=(Y'_e,D'_e)$.
Now $Y'$ contains infinitely many internal $(-2)$-curves given by the strict transforms of the $(-1)$-curves in $Y_e$ passing through $p$.
\end{example}

\begin{example}\label{W_trivial_Adm_infinite}
We describe an example of a Looijenga pair $(Y,D)$ such that $W$ is trivial and $\Adm$ is infinite.

Let $(\oY,\oD)$ be the toric Looijenga pair given by $\bF_1$ together with its toric boundary.
Label the boundary divisors $\oD_1,\ldots,\oD_4$ so that $\oD_1^2=-1$, $\oD_2^2=0$, $\oD_3^2=1$, and $\oD_4^2=0$.
Let $(\oY',\oD')$ be the toric pair obtained from $(\oY,\oD)$ by the following sequence of toric blowups. 
We first blowup $\oD_1 \cap \oD_2$, $\oD_2 \cap \oD_3$, and $\oD_3 \cap \oD_4$, then blowup the intersection point of the strict transform of $\oD_4$ with the exceptional divisor over $\oD_3 \cap \oD_4$. 
Now let $(Y,D)$ be the Looijenga pair given by performing an interior blowup at a point of each of the $(-1)$-curves contained in $\oD'$. 
Then $D$ is a cycle of eight $(-2)$-curves. One can check that $D^{\perp}$ does not contain any classes $\alpha$ such that $\alpha^2=-2$. Thus $\Phi= \emptyset$ and $W$ is trivial for $(Y,D)$.
Moreover, choosing the positions of the interior blowups appropriately (so that $(Y,D)=(Y_e,D_e)$), there is an elliptic fibration $f \colon Y \rightarrow \bP^1$ with $f^{-1}(\infty)=D$. 
There are no reducible fibers of $f$ besides $D$ (because $\Phi = \emptyset$ and $f$ is relatively minimal). It follows that the Mordell-Weil group of $f$ is infinite.
(Indeed, writing $\eta \in \bP^1$ for the generic point, the Mordell--Weil group of sections of the elliptic fibration $f$ is given by
$$\MW(f)=\Pic^0(Y_{\eta})=\langle D \rangle^{\perp}/ \langle \Gamma \ | \ f_*\Gamma = 0 \rangle$$
$$= \langle D \rangle^{\perp}/\langle D_1,\ldots,D_8 \rangle.$$
In particular, $\rk \MW(f) = 1$.) Thus the group $\Aut(Y,D)$ is infinite.
Now by Theorem~\ref{exactseqtheorem} we find that $\Adm$ is infinite. 

\end{example}

Recall from Definition~\ref{markingdef} that
if $((Z,D),p_i)$ is a Looijenga pair with marked boundary,
and $\mu: \Pic(Y) \to \Pic(Z)$ is a marking of $\Pic(Z)$, the
marked period point of $((Z,D),p_i,\mu)$ 
is a point in 
\[
T_Y := \Hom(\Pic(Y),\Pic^0(D)).
\]

\begin{construction} \label{universalfamilies} 
\emph{Universal families}. Let $(Y,D)$ be
a Looijenga pair, and $\pi: Y \to \oY$ a toric model, with 
exceptional divisors $\{E_{ij}\}$ which are disjoint interior
$(-1)$-curves. Varying $\phi \in T_Y$, the construction of
Proposition \ref{closetotorelli} produces 
sections 
$p_i: T_Y \to T_Y \times D_i^o \subset T_Y \times \oY$,
and then unique sections $q_{ij}:T_Y \to T_Y \times D_i^o$ such
that 
\[
\phi(E_{ij}) = \cO_D(q_{ij}(\phi))^{-1} \otimes \cO_D(p_i(\phi)) \in 
\Pic^0(D) = \bG_m.
\]
Explicitly, let $p_i$ be the section $\op_i$ of Lemma \ref{toriccasefamilies} (this
involves choosing the right inverse $\gamma$ of $\psi$, but see Remark \ref{sectionindremark}),
then $q_{ij}(\phi) \in \bG_m$ is the point
\[
\phi(E_{ij})^{-1} \cdot p_i(\phi) \in D_i^o,
\]
where $\Pic^0(D)=\Gm$ acts on $D_i^o$ using the convention
of Lemma \ref{pic0D}.

Let $\Pi:(\cY_{\{E_{ij}\}},\cD) \to T_Y \times \oY$ be the iterated
blowup along the sections 
$$
q_{ij} \subset T_Y \times D_i^o \subset T_Y \times \oY.
$$
This comes with a 
marking $\mu: \Pic(Y) \to \Pic(\cY)$ preserving boundary
classes, and sending $E_{ij}$ to the corresponding exceptional
divisor $\cE_{ij}$. 
This induces a marking of $\Pic(Z)$ for each fibre $Z$. We call
$\lambda:(\cY_{\{E_{ij}\}},p_i,\mu) \to T_Y$ a 
{\it universal family}. See Theorem~\ref{modulistacks} for justification of this term.

If $\tau: Y \to Y'$ is a toric blowup, with exceptional divisor $E$, 
and $Y$ has a toric model as above, then there is a divisorial
contraction $\tilde\tau: \cY_{\meij} \to \tcY_{\meij}'$ which blows down
the $(-1)$-curve $\mu(E)$ in each fibre --- this is a family
of toric blowups. Observe that identifying $\Pic(Y)$, $\Pic(Y')$ with
$A_1(Y)$, $A_1(Y')$ respectively, we have a map $\tau_*:A_1(Y)\rightarrow
A_1(Y')$, and hence a transpose map $T_{Y'}=\Hom(A_1(Y'),\Gm)
\rightarrow \Hom(A_1(Y),\Gm)=T_Y$, an inclusion of tori. This identifies
$T_{Y'}$ with the elements of $T_Y$ which take the value $1$
on exceptional divisors of $\tau$.
We define $\lambda':\cY'_{\meij} \to T_{Y'}$ to be the restriction
of $\tcY_{\meij}'$ to $T_{Y'}\subset T_Y$. 
This inherits markings of the boundary and
the Picard group. In this way we have a universal family associated
with each configuration of exceptional curves for a toric model
of some toric blowup.
\end{construction}

\begin{remark}
\label{sectionindremark}
Note in the construction we made a choice of right inverse
$\gamma \colon T_{\oY} \to \Aut^0(D)$ of $\psi$. By Proposition
\ref{basicseq},
any two choices differ by a homomorphism
$h \colon T_{\oY} \to \Aut(\oY,D)$. 
One can check that $h$ together with the action of $\Aut(\oY,D)$ on $\oY$ 
induces a canonical identification of the universal families constructed.
\end{remark}

\begin{remark} There are in general infinitely many universal
families of a given combinatorial type. For a given pair $(Y,D)$
with exceptional divisors $E_{ij}$ for a toric model, the above construction
gives a finite number of families, as there is a choice of order of blowup. 
However, there may be an
infinite number of sets of exceptional divisors of the same combinatorial
type, giving rise to an infinite number of families.
We will see that any two are birational, canonically
identified by a birational map, see Construction-Theorem \ref{biratmap}.
\end{remark}

By construction:
\begin{lemma} \label{inverselemma} For $\phi \in T_Y$, the marked period point of
the fibre $((\cY,\cD),p_i,\mu)_{\phi}$ of 
a universal family is $\phi$. 
\end{lemma}

In particular, the fiber of a universal family $(\cY,\cD)/T_{Y_0}$ over the identity $e \in T_{Y_0}$  is the pair $(Y_e,D)$ defined above.

\begin{corollary} \label{genericlocusunivfam} The locus of generic pairs in a universal family $(\cY,\cD)/T_{Y}$ is the complement in $T_{Y}$ of the countable union of hypertori 
$$T_{\alpha}=\{ \phi \in T_{Y} \ | \ \phi(\alpha)=1 \}$$
for $\alpha \in \Phi$.
\end{corollary}
\begin{proof}
This follows from Corollary~\ref{genericlocus}. 
\end{proof}

We construct a birational action of $\Adm_{Y_0}$ on a universal family. This action is used to identify $\Adm_{Y_0}$ with the monodromy group, see Theorem~\ref{strong_monodromy}.

\begin{thmconst} \label{biratmap}
Let $(Y_0,D)$ be a generic Looijenga pair.
Let $\meij,\mfij$ be two exceptional configurations for $(Y_0,D)$, not necessarily of the same type. Then
there is a canonical birational map 
$$
\cY_{\meij} \dasharrow \cY_{\mfij},
$$
commuting with the projections to $T_{Y_0}$. This birational
map restricts to an isomorphism over a Zariski open set of $T_{Y_0}$ containing the locus of generic pairs, and respects the markings of the Picard group and the boundary of each fiber over this locus.
\end{thmconst}

\begin{proof}
Suppose first the configurations are on $Y_0$ (rather than on possibly different toric blowups of $Y_0$). 
Let $U \subset T_{Y_0}$ be the maximal open subset such that each $F_{ij} \subset Y_0$  deforms to a family of $(-1)$-curves $\cF_{ij} \subset \cY_{\meij}|_U$ over $U$.
The Zariski open set $U$ contains the locus of generic fibers by Lemma~\ref{lim=exc_for_gen}. Write $\cY'=\cY_{\meij}|_U$.
The $\cF_{ij}$ restrict to an exceptional configuration on each fiber of $\cY'$ over $U$, and we have a birational morphism 
$(\cY',\cD') \rightarrow (\ocY',\ocD')$ given by blowing down these families of curves. Let $ \pi \colon (Y_0,D) \rightarrow (\oY_0,\oD)$ be the toric model obtained by contracting the $F_{ij}$.
Then $(\ocY',\ocD')$ is a fiber bundle over $U$ with fiber $(\oY_0,\oD)$.
Also, $\cY' \to \ocY'$ restricts to an isomorphism $\cD' \rightarrow \ocD'$, so the markings $p^E_i \colon T_{Y_0} \rightarrow \cY_{\meij}$ induce markings of $\ocD' \subset \ocY'$.
Let $p^F_i$ be the sections of the trivial family $U \times (\oY_0,\oD)$ given by Construction~\ref{universalfamilies} for $\cY_{\mfij}$.
By construction, the marked period points of the fibers of the families $((\ocY',\ocD'),p^E_i)$ and $(U \times (\oY_0,\oD),p^F_i)$ over each $\phi \in U$ coincide. 
So, by Lemma~\ref{toriccase}, there is a unique isomorphism $f: ((\ocY',\ocD'),p^E_i) \to (U \times (\oY_0,\oD),p^F_i)$ over $U$.
Each of $\ocY'$ and $U \times \oY$ comes with sections $q_{ij}$ (given in the first case by the images of the exceptional divisors $\cF_{ij}$, in the second by Construction~\ref{universalfamilies} for $\cY_{\mfij}$), 
which are identified under the isomorphism $f$. 
Thus after performing the iterated blow-up of the $q_{ij}$ on $\ocY'$ and $U \times \oY_0$ respectively, $f$ induces an isomorphism $\cY' \rightarrow \cY_{\mfij}|_U$. 
That is, we obtain a birational map $\cY_{\meij} \dasharrow \cY_{\mfij}$ which is an isomorphism over $U$.

If the configurations are on toric blowups of $Y_0$ we
make the obvious modifications: we replace
$Y_0$ by a toric blowup $\tau: Z_0 \to Y_0$ 
on which they both appear and carry out
the above. Then we restrict the birational maps to the 
subtorus $\Hom(\Pic(Y_0),\bG_m) \subset \Hom(\Pic(Z_0),\bG_m)$,
and obtain induced birational maps between the universal
families (which we recall are obtained from the restricted families
via the families of toric blowdowns determined by $\tau$).
\end{proof}

\begin{construction}
\label{Waction}
Let $(Y_0,D)$ be a generic Looijenga pair.
Observe that
$\Aut(\Pic(Y_0))$ acts by precomposition on $T_{Y_0}$:
$$
g(\phi) := \phi \circ g^{-1}.
$$
If $g$ is admissible and $\{E_{ij}\}$ is an exceptional collection,
then $\{g(E_{ij})\}$ is an exceptional collection necessarily of
the same combinatorial type as $\{E_{ij}\}$. This induces, by
the construction of the universal families, 
a commutative diagram
$$
\begin{CD}
\cY_{\meij} @>>> \cY_{\{g(E_{ij})\}} \\
@VVV  @VVV \\
T_{Y_0} @>{\phi \mapsto \phi \circ g^{-1}}>> T_{Y_0} 
\end{CD}
$$
where the horizontal maps are isomorphisms. Composing
$\cY_{\meij}\rightarrow \cY_{\{g(E_{ij})\}}$ with the canonical birational map 
$\cY_{\{g(E_{ij})\}}\dasharrow \cY_{\meij}$ gives a birational map
$$
\psi_g: \cY_{\meij} \dasharrow \cY_{\meij}
$$
which is an isomorphism over a Zariski open set containing the locus of generic pairs. 
In particular this gives a canonical action of $\Adm_{Y_0}$ on $\cY_{\meij}$
by birational automorphisms. By construction the composition 
$$
\begin{CD}
\Pic(Y_0) @>{r^{-1}}>> \Pic(\cY_{\meij}) @>{\psi_{g*}}>> \Pic(\cY_{\meij})
@>{r}>> \Pic(Y_0)
\end{CD}
$$
is $g \in \Aut(\Pic(Y_0))$ (here $r$ is the restriction).
\end{construction}

\begin{example}
Consider the pair $(Y,D)$ obtained by blowing up one general point
on each coordinate axis of $\PP^2$, with $D$ the proper transform
of the toric boundary of $\PP^2$. Write the generators of $\Pic(Y)$
as $L,E_1,E_2,E_3$ with $L$ the pull-back of a line in $\PP^2$
and the $E_i$'s the exceptional divisors. Then $\{E_1,E_2,E_3\}$
is an exceptional configuration, as is $\{F_1,F_2,F_3\}$
where $F_i=(L-E_1-E_2-E_3)+E_i$.
We obtain universal families $\cY_{\meij},\cY_{\mfij}\rightarrow
T_Y$. The birational map constructed above $f:\cY_{\meij}\dasharrow
\cY_{\mfij}$ is an isomorphism away from the locus where the three
blown-up points lie on a line $\oL$. Over such a point, 
the curve of class $F_i$ decomposes as a union of irreducible
curves of class $\alpha := L-E_1-E_2-E_3$ and $E_i$.
The curve of class $\alpha$ is the proper transform of $\oL$
and is common to all three curves, hence the three curves cannot be
simultaneously contracted. The proper transform of $\oL$ must be flopped
before this contraction can be performed.

Note in this example that $D^{\perp}$ is generated by $\alpha$, and 
$\Phi=\{\pm \alpha\}$. The reflection $s_{\alpha}$ satisfies
$s_{\alpha}(E_i)=F_i$. It is an admissible automorphism, and $W=\{\id, 
s_{\alpha}\}$. Since the only non-trivial automorphism which preserves
the boundary classes and the intersection pairing is $s_{\alpha}$, it is
clear that $W=\Adm_Y$.
\end{example}

Using the construction of universal families together with the $\Adm_{Y}$ action, we show that $\Adm_{Y}$ is equal to the monodromy group.

\begin{theorem} \label{strong_monodromy}
Let $(Y,D)$ be a Looijenga pair. The group $\Adm_Y$ is the monodromy group in the following sense.
Let $(\cY,\cD) \rightarrow S$ be a family of Looijenga pairs over a connected base $S$ together with a point $s \in S$ and an identification 
$$(Y,D) \stackrel{\sim}{\longrightarrow} (\cY_s,\cD_s).$$
Then the monodromy map
$$\rho \colon \pi_1(S,s) \rightarrow \Aut(\Pic(Y))$$
has image contained in $\Adm_Y$. 
Furthermore, in the analytic category, there exists a family as above such that $S$ is smooth and the image of $\rho$ is equal to $\Adm_Y$.
\end{theorem}

\begin{proof}
We have already established that the monodromy group of any family is contained in $\Adm_Y$. See Lemma~\ref{MonAdmPhi}. It remains to show that there is a family with smooth base and monodromy group equal to $\Adm_Y$.

Let $\lambda \colon (\cY,\cD) \rightarrow S$ be a choice of universal family, with marking
$$\mu \colon \Pic(Y) \times S \rightarrow R^2\lambda_*\bZ$$
and action $\psi$ of $\Adm_Y$. Here $S=\Hom(\Pic(Y),\bG_m)$.
Let $A \subset \Pic(Y)_{\bR}$ be a connected open cone on which $\Adm_Y$ acts properly discontinuously. For example we can take $A=C^+$. 
Working in the analytic topology, let $\Omega \subset S$ denote the \emph{tube domain} associated to $A$. That is
$$\Omega=(\Pic(Y)_{\bR} + iA)/\Pic(Y) \subset (\Pic(Y) \otimes_{\bZ} \bC) / \Pic(Y) = S,$$
an open analytic subset of $S$. (Here we have used the identification $\Pic(Y)=\Pic(Y)^*$ given by Poincar\'e duality.)
Then $\Adm_Y$ acts properly discontinuously on $\Omega$. (Indeed, the real analytic morphism
$$\Omega \rightarrow A$$
given by projection onto the imaginary part is $\Adm_Y$ equivariant and $\Adm_Y$ acts properly discontinuously on $A$.)
Let
$$\Omega^o := \Omega \setminus \bigcup_{g \in \Adm_Y} \Fix(g)$$
denote the complement of the fixed loci of the elements of $\Adm_Y$. Note that $$\bigcup_{g \in \Adm_Y} \Fix(g) \subset \Omega$$ is a locally finite union of analytic subvarieties because the action is properly discontinuous. 
Hence $\Omega^o \subset \Omega$ is a connected open analytic subset. Note also that $\Omega^o$ is contained in the locus of generic pairs. 
Indeed, the locus of generic pairs is the complement of union of the hypertori $T_{\alpha}$ for $\alpha \in \Phi$ by Corollary~\ref{genericlocusunivfam}, 
and $T_{\alpha}=\Fix(s_{\alpha})$ where $s_{\alpha} \in W \subset \Adm_Y$ is the reflection in the root $\alpha$.
Let
$$U:=\Omega^o/\Adm_Y$$
be the quotient of $\Omega^o$ by $\Adm_Y$, a complex analytic manifold.
Let $(\cY_U,\cD_U) \rightarrow U$ be the family of Looijenga pairs over $U$ given by the quotient of the restriction of the family $(\cY,\cD) \rightarrow S$ to $\Omega^o$. 
(Note that the birational action of $\Adm_Y$ on the universal family is biregular over the locus of generic pairs and hence over $\Omega^o$. See Construction~\ref{Waction}.)
Let $t \in \Omega^o$ be a basepoint, and  
$u \in U$ the image of $t$. The Galois covering map $\Omega^o \rightarrow U$ with group $\Adm_Y$ corresponds to a surjection
$$\pi_1(U_,u) \rightarrow \Adm_Y.$$
Given $g \in \Adm_Y$, let $[\gamma] \in \pi_1(U,u)$ be a lift of $g$. Then $\gamma$ is a loop based at $u \in U$ which lifts to a path $\tilde{\gamma}$ in $\Omega^o$ from $t$ to $g^{-1}t$. 
Now the monodromy transformation associated to the loop $\gamma$ for the family $(\cY_U,\cD_U) \rightarrow U$ is identified with 
$g \in \Adm_Y \subset \Aut(\Pic(Y))$ via the marking isomorphism
$$\mu_t \colon \Pic(Y) \stackrel{\sim}{\longrightarrow} \Pic(\cY_t).$$ 
\end{proof}

\section{Moduli stacks}

We give a complete description of the moduli stacks of Looijenga pairs, with and without markings. Note that these stacks are highly non-separated in general. The situation is very similar to that of moduli of K3 surfaces without polarization, cf. \cite{LP80}, \S10.

We work in the analytic category. The stacks we define are stacks over the category of analytic spaces. 

Fix a Looijenga pair $(Y_0,D)$. Let $\tilde{\mathbb{M}}_{Y_0}$ denote the moduli stack of families of Looijenga pairs $(Y,D)$ together with a marking of $D$ and a marking of $\Pic(Y)$ by $\Pic(Y_0)$. More precisely, for an analytic space $S$, the objects of the category $\tilde{\mathbb{M}}_{Y_0}(S)$ are morphisms
$$\lambda \colon (\cY,\cD=\cD_1+\cdots+\cD_n) \rightarrow S$$
together with an isomorphism
$$\mu \colon \Pic(Y_0) \times S \stackrel{\sim}{\longrightarrow} R^2\lambda_*\bZ$$
and sections
$$p_i \colon S \rightarrow \cD_i$$
of $\cD_i \rightarrow S$
such that 
\begin{enumerate}
\item $\cY/S$ is a flat family of surfaces,
\item $\cD_i$ is a Cartier divisor on $\cY/S$ for each $i$, and
\item Each closed fiber $(\cY_s,\cD_s,\mu_s,\{p_i(s)\})$ is a Looijenga pair together with marking $\mu_s \colon \Pic(Y_0) \rightarrow \Pic(\cY_s)$ of the Picard group and marking $p_i(s) \in \cD_{i,s}$ of the boundary. 
\end{enumerate}
Furthermore, in the cases $n=1$ or $2$, we assume given an orientation of $\cD$, that is, an identification
$$\bZ \times S \stackrel{\sim}{\longrightarrow} R^1\lambda_*\bZ_{\cD}.$$
The morphisms in the category from $(\cY,\cD)/S$ to $(\cY',\cD')/S'$ over a morphism $S \rightarrow S'$ are isomorphisms
$$(\cY,\cD) \stackrel{\sim}{\longrightarrow} (\cY',\cD') \times_{S'} S$$
over $S$ compatible with the markings and the orientation.

Similarly, let $\tilde{\mathbb{M}}'_{Y_0}$ denote the moduli stack of Looijenga pairs $(Y,D)$ together with a marking of $\Pic(Y)$ by $\Pic(Y_0)$, 
$\mathbb{M}_{Y_0}$ the moduli stack of Looijenga pairs with a marking of $D$, and $\mathbb{M}'_{Y_0}$ the moduli stack of Looijenga pairs.

We have the period mapping
$$\tilde{\mathbb{M}}_{Y_0} \rightarrow T_{Y_0}$$
$$((Y,D),\mu,\{p_i\}) \mapsto \phi_Y.$$
(Note: If $(\cY,\cD)/S$ is an object of $\tilde{\mathbb{M}}_{Y_0}$ then the sections $p_i$ and the orientation of $\cD$ determine a canonical isomorphism 
$$D \times S \stackrel{\sim}{\longrightarrow} \cD.$$
This is used to define the period mapping for families over an arbitrary base $S$.)
Similarly, writing $T_{Y_0}'=\Hom(D^{\perp},\bG_m)$, we have period mappings
$$\tilde{\mathbb{M}}'_{Y_0} \rightarrow T_{Y_0}',$$
$$\mathbb{M}_{Y_0} \rightarrow [T_{Y_0}/\Adm],$$
and
$$\mathbb{M}'_{Y_0} \rightarrow [T_{Y_0}'/\Adm].$$
(Here for a group $G$ acting on an analytic space $X$ we write $[X/G]$ for the stack quotient. We also write $\Adm=\Adm_{Y_0}$ for brevity.)

Let $\Sigma$ denote the set of connected components of the complement
$$C^{++}_D \setminus \bigcup_{\alpha \in \Phi} \alpha^{\perp}.$$ 
(So $\Sigma$ is permuted simply transitively by the Weyl group $W$.)
Let $U=T_{Y_0} \times \Sigma$. Define an \'etale equivalence relation $R$ on $U$ as follows: $(p,\sigma) \sim (p,\sigma')$ iff $\sigma$ and $\sigma'$ are contained in the same connected component of $C^{++}_D \setminus \bigcup_{\alpha \in \Phi_p} \alpha^{\perp}$, where
$$\Phi_p := \{\alpha \in \Phi \ | \ p(\alpha)=1  \}.$$
Let $\tilde{T}_{Y_0}$ denote the analytic space $U/R$. 
Thus we have an \'etale morphism $\tilde{T}_{Y_0} \rightarrow T_{Y_0}$ given by the first projection $U= T_{Y_0} \times \Sigma \rightarrow T_{Y_0}$, which is an isomorphism over the open analytic set $T_{Y_0} \setminus \overline{\bigcup_{\alpha \in \Phi} (p(\alpha)=1)}$. Note that $\tilde{T}_{Y_0}$ is not separated if $\Phi \neq \emptyset$.
We define $\tilde{T}_{Y_0}' \rightarrow T_{Y_0}'$ similarly. That is, $\tilde{T}_{Y_0}'=U'/R'$ where $U'=T_{Y_0}' \times \Sigma$ and $R'$ is defined by the same rule as above.

Recall the action of $\Aut^0(D)=\bG_m^n$ on $$T_{Y_0}=\Hom(\Pic(Y_0),\bG_m)=\Hom(\Pic(Y_0),\Pic^0(D))$$ from Lemma~\ref{funny}.

Let $K=\Hom(N',\bG_m)$ where $N'$ is the group defined in Theorem~\ref{torelliI}. Every object $(\cY,\cD)/S$ of $\tilde{\mathbb{M}}'_{Y_0}(S)$ has a canonical subgroup $K \times S$ of its automorphism group. (These are the automorphisms acting trivially on the Picard group of the fibers). See Theorem~\ref{exactseqtheorem}. Let $\tilde{\mathbb{M}}''_{Y_0}$ denote the rigidification of $\tilde{\mathbb{M}}'_{Y_0}$ along $K$ in the sense of \cite{ACV05}, \S 5. (Thus the objects of $\tilde{\mathbb{M}}''_{Y_0}$ and $\tilde{\mathbb{M}}'_{Y_0}$ coincide locally, but the automorphism group in $\tilde{\mathbb{M}}''_{Y_0}$ is the quotient of the automorphism group in $\tilde{\mathbb{M}}'_{Y_0}$ by $K$.) Similarly let $\mathbb{M}''_{Y_0}$ denote the rigidification of $\mathbb{M}'_{Y_0}$ along $K$. (Note that the monodromy group $\Adm$ acts trivially on $K$.) The period mappings $\tilde{\mathbb{M}}'_{Y_0} \rightarrow T_{Y_0}'$ and $\mathbb{M}'_{Y_0} \rightarrow [T_{Y_0}'/\Adm]$ descend to maps $\tilde{\mathbb{M}}''_{Y_0} 
\rightarrow T_{Y_0}'$ and $\mathbb{M}''_{Y_0} \rightarrow [T_{Y_0}'/\Adm]$.

\begin{theorem}\label{modulistacks}
We have identifications $$\tilde{\mathbb{M}}_{Y_0}=\tilde{T}_{Y_0},$$ $$\tilde{\mathbb{M}}'_{Y_0}=[\tilde{T}_{Y_0}/\Aut^0(D)],$$ $$\mathbb{M}_{Y_0}=[\tilde{T}_{Y_0}/\Adm],$$ $$\mathbb{M}'_{Y_0}=[\tilde{T}_{Y_0}/\Aut^0(D) \times \Adm],$$ $$\tilde{\mathbb{M}}''_{Y_0}=\tilde{T}_{Y_0}',$$ and
$$\mathbb{M}''_{Y_0}=[\tilde{T}_{Y_0}'/\Adm],$$ 
compatible with the period mappings.
\end{theorem}
\begin{proof}
The identification $$\tilde{\mathbb{M}}_{Y_0} \stackrel{\sim}{\longrightarrow} \tilde{T}_{Y_0}$$ is obtained as follows.
Given $((\cY,\cD)/S,\mu,\{p_i\}) \in \tilde{\mathbb{M}}_{Y_0}(S)$ we have the associated period mapping $\phi \colon S \rightarrow T_{Y_0}$, see Definition~\ref{markingdef}(3).
We define a lift $\tilde{\phi} \colon S \rightarrow \tilde{T}_{Y_0}$ of $\phi$ by
$$\tilde{\phi}(s)=(\phi(s), \sigma)$$
where $\sigma \subset \mu_s^{-1}(\Nef(\cY_s))$, for each $s \in S$. 
Note that $\mu_s^{-1}(\Nef(\cY_s))$ is the closure of a connected component of $C^{++}_D \setminus \bigcup_{\alpha \in \Phi_p} \alpha^{\perp}$ by Lemma~\ref{Nefcones} and Theorem~\ref{thWeylchambers}.
So $\sigma$ is uniquely determined up to the equivalence relation $R$ and $\tilde{\phi}$ is a well-defined map to $\tilde{T}_{Y_0}=U/R$.

We now establish that the morphism $\tilde{\mathbb{M}}_{Y_0} \rightarrow \tilde{T}_{Y_0}$ is an isomorphism.
Objects of $\tilde{\mathbb{M}}_{Y_0}$ have no non-trivial automorphisms by Proposition~\ref{basicseq}. So the stack $\tilde{\mathbb{M}}_{Y_0}$ is (represented by) an analytic space. 
The period mapping $\tilde{\mathbb{M}}_{Y_0} \rightarrow T_{Y_0}$ is \'etale by \cite{L81}, II.2.5. Hence also $\tilde{\mathbb{M}}_{Y_0} \rightarrow \tilde{T}_{Y_0}$ is \'etale.
The map $\tilde{\mathbb{M}}_{Y_0} \rightarrow \tilde{T}_{Y_0}$ is injective on points by the global Torelli theorem for Looijenga pairs, Theorem~\ref{torelliI}.
Indeed, two marked Looijenga pairs $((Y,D),\mu,\{p_i\})$, $((Y',D'),\mu',\{p_i'\})$ are isomorphic iff $\phi_Y=\phi_{Y'}$ and $\mu^{-1}(\Nef(Y))=\mu'^{-1}(\Nef(Y'))$ by Remark~\ref{torellinefcone}. 
Also, the map is surjective on points by the construction of universal families (see Lemma~\ref{inverselemma}) and the fact that connected components of 
\mbox{$C^{++}_D \setminus \bigcup_{\alpha \in \Phi_p} \alpha^{\perp}$} are permuted transitively by the Weyl group $W(\Phi_p)$.
Hence the map $\tilde{\mathbb{M}}_{Y_0} \rightarrow \tilde{T}_{Y_0}$ is an isomorphism as claimed.

The remaining identifications follow by passing to the quotients corresponding to forgetting the marking of $\Pic(Y)$ and/or $D$. Note that applying $\Hom(\cdot,\bG_m)$ to the exact sequence
$$0 \rightarrow D^{\perp} \rightarrow \Pic(Y_0) \rightarrow \bZ^n \rightarrow N' \rightarrow 0$$
we obtain the exact sequence
$$1 \rightarrow K \rightarrow \bG_m^n \rightarrow T_{Y_0} \rightarrow T_{Y_0}' \rightarrow 1.$$
Hence $$\tilde{\mathbb{M}}'_{Y_0}=[\tilde{T}_{Y_0}/\Aut^0(D)]=[\tilde{T}_{Y_0}/\bG_m^n],$$ 
where $K \subset \bG_m^n$ acts trivially and the quotient $H:=\bG_m^n/K$ acts freely. Thus rigidifying $\tilde{\mathbb{M}}'_{Y_0}$ along $K$ yields $[\tilde{T}_{Y_0}/H]=\tilde{T}_{Y_0}'$.
\end{proof}

\section{Generalization of the Tits cone}

In this section we explore to what extent some additional constructions from \cite{L81} extend to the more general context of this paper.

The paper \cite{L81} considers Looijenga pairs $(Y,D)$ such that the following conditions are satisfied:
\begin{assumptions}\label{looijenga_assumptions}
\begin{enumerate}
\item The number $n$ of irreducible components of $D$ is less than or equal to $5$.
\item The intersection matrix $(D_i \cdot D_j)$ is negative semi-definite.
\item There do not exist $(-1)$-curves contained in $D$.
\end{enumerate}
\end{assumptions}
(Note that condition $(3)$ is not essential: Under condition (2), there is always a toric blowdown $(Y,D) \rightarrow (Y',D')$ such that $(Y',D')$ satisfies (3).)

\begin{remark}
Under assumptions~\ref{looijenga_assumptions}, Looijenga gives an explicit description of the set $\Delta$ and shows that it is a basis of the lattice 
$D^{\perp}:=\langle D_1,\ldots,D_n \rangle^{\perp}$. In general however the set $\Delta$ does not give a basis of $D^{\perp}$. In fact $\Delta$ may be infinite, see Example~\ref{Delta_infinite}. 
At the other extreme, there are examples with $D^{\perp} \neq 0$ and $\Delta = \emptyset$, see e.g. \cite{F13}, Examples 4.3 and 4.4.
\end{remark}

Under assumptions~\ref{looijenga_assumptions}, Looijenga defines the Tits cone $I \subset \Pic(Y)_\bR$ as follows.
(Here we use our notation.)
Write $\Delta=\Delta_{Y_e}$ for the set of classes of $(-2)$-curves on $Y_e$. Define the fundamental chamber
$$C=\{ x \in C^+ \ | \ x \cdot \alpha > 0 \mbox{ for all } \alpha \in \Delta \}.$$
The Tits cone $I$ is defined by
$$I = \bigcup_{w \in W} w(\overline{C}).$$
Looijenga proves that the Weyl group $W$ acts properly discontinuously on the interior $\Int(I)$ of $I$ \cite{L81}, Corollary~1.14.
Moreover, the reflection hyperplanes $\alpha^{\perp} \subset \Pic(Y)_{\bR}$, $\alpha \in \Phi$ are dense in $\Pic(Y)_{\bR} \setminus \Int(I) \cup (-\Int(I))$
by \cite{L81}, Theorem~II.1.5. So $$\Int(I) \cup (-\Int(I)) \subset \Pic(Y)_{\bR}$$ is the maximal $W$-equivariant open set on which $W$ acts properly discontinuously.

In the general case recall that we have an inclusion $W \subset \Adm$. The group $\Adm$ is the full monodromy group and the Weyl group $W$ is the normal subgroup given by Picard--Lefschetz transformations. Under assumptions~\ref{looijenga_assumptions} we have $W=\Adm$, see Lemma~\ref{looadm}. In general $W \neq \Adm$, and in fact 
the index of $W \subset \Adm$ may be infinite, see Example~\ref{nonfgcone}. Moreover, there are examples such that $W$ is trivial and $\Adm$ is infinite, see Example~\ref{W_trivial_Adm_infinite}.

However, we show that the fact that $W=\Adm$ acts properly discontinuously on the Tits cone admits a generalization as follows.

\begin{proposition}\label{Titscone}
Assume the conditions~\ref{looijenga_assumptions}. Let $(Y_g,D)$ denote a generic deformation of $(Y,D)$. 
Then the closure of the Tits cone $I \subset \Pic(Y)_{\bR}$ is equal to the closure of $$\overline{\NE}(Y_g)+\langle D_1,\ldots,D_n \rangle_{\bR}.$$
\end{proposition}
\begin{proof}
Write $A=\overline{\NE}(Y_g)+\langle D_1,\ldots,D_n \rangle_\bR$, a convex cone in $\Pic(Y)_{\bR}$.
We must show that $\overline{A}=\overline{I}$.

The Weyl group $W$ acts transitively on the set of $(-1)$-curves on $Y_g$ meeting $D_i$, for each $i=1,\ldots,n$, by \cite{L81} Theorem I.4.6.
Also $C^+ \subset I$ by \cite{L81}, Lemma~I.3.7. 

If $D^2<0$ then $I$ is the convex hull of the union of $\langle D_1,\ldots,D_n \rangle_{\bR}$ and the set of $(-1)$-curves on $Y_g$ by \cite{L81} Proposition~I.3.9, the description of the extremal facets of $I$ in \S I.3.8, and Theorem I.4.6. Now by the description of $\overline{\NE}(Y_g)$ given by Lemma~\ref{Moricones} we deduce that $\overline{I}=\overline{A}$ if $D^2<0$. 

If $D^2=0$ then 
$$I=\{x \in \Pic(Y)_{\bR} \ | \ x \cdot D > 0 \} \cup \langle D_1,\ldots,D_n\rangle_{\bR}$$ 
by \cite{L81} Proposition~I.3.9. It is easy to see that $\overline{I}=\overline{A}$ in this case. Indeed, if $x \cdot D > 0$ then $(x+ND) \in C^+$ for $N \gg 0$, thus $x \in A$. So $I \subset A$. Conversely, $A \subset \overline{I}$ because $D$ is effective and $D \cdot D_i = 0$ for each component $D_i$ of $D$ (note that $D$ is either irreducible with $D^2=0$ or a cycle of $(-2)$-curves).
\end{proof}

\begin{proposition}\label{discontinuous_action_Mori_cone}
Let $Y$ be a smooth projective surface. Let $\Gamma \subset \Aut(\Pic(Y))$ be a subgroup preserving the semigroup of effective classes. Then $\Gamma$ acts properly discontinuously on the interior of the cone $\overline{\NE}(Y)$.
\end{proposition}
\begin{proof}
First note that any subgroup $\Gamma$ of $\Aut(\Pic(Y))$ acts properly discontinuously on the positive cone $C^+$.

Now assume as in the statement that $\Gamma$ preserves the semigroup of effective classes. We will use the Zariski decomposition of effective divisors on the surface $Y$ to show that $\Gamma$ acts properly discontinuously on the interior of the effective cone.
Let $D$ be a pseudoeffective $\bR$-divisor on the surface $Y$ (that is, $D \in \overline{\NE}(Y)$). Then there is a unique decomposition
$$D=P+N$$
where $P$ and $N$ are $\bR$-divisors, $P$ is nef, $N$ is effective, and, writing $N=\sum a_iN_i$ where $N_i$ is irreducible and $a_i \in \bR_{>0}$ for each $i$, we have $P \cdot N_i = 0$ for each $i$ and the matrix $(N_i \cdot N_j)$ is negative definite. See \cite{KMM87}, Theorem 7.3.1.
Moreover $D$ lies in the interior of the effective cone iff $P^2>0$.
(Indeed $C^+ \subset \overline{\NE}(Y)$ so $P^2>0$ implies $D$ lies in the interior of $\overline{\NE}(Y)$. Conversely if $P^2=0$ and $P \neq 0$ then $D \in P^{\perp}$ and $P$ is nef so $D$ does not lie in the interior of $\overline{\NE}(Y)$. Finally if $P=0$ then we can find a nef divisor $B$ such that $B \cdot D =0$ --- take an ample divisor $A$ and write $B=A+\sum \lambda_i N_i$ such that $B \cdot N_i=0$ for each $i$, then $\lambda_i > 0$ for each $i$ and $B$ is nef, $B \cdot D=0$. Thus $D$ does not lie in the interior of $\overline{\NE}(Y)$.)
Note also that if $D$ lies in the interior of the effective cone then the Zariski decomposition $D=P+N$ is characterized by the following properties: $P$ is nef, $P \neq 0$, $N$ is effective, and $P \cdot N = 0$. (Indeed, writing $N=\sum_{i=1}^k a_iN_i$ as above, $P \cdot N = 0$ and $P$ nef implies $P \cdot N_i =0$ for each $i$. Also $P^2>0$ because $D$ lies in the interior of the effective cone, so the subspace $\langle N_1,\ldots,N_k \rangle_{\bR} \subset P^{\perp}$ is negative definite.
It remains to show that the $N_i$ are linearly independent. Otherwise, we have a nontrivial expression $\sum \alpha_iN_i = \sum \beta_iN_i$ where $\alpha_i, \beta_i \in \bR_{\ge 0}$ and $\alpha_i\beta_i=0$ for each $i$. But then 
$$0>(\sum \alpha_i N_i)^2=(\sum \alpha_i N_i) \cdot (\sum \beta_i N_i) \ge 0,$$ 
a contradiction.) 

Let $B \subset \Pic(Y)_{\bR}$ denote the interior of the effective cone. We need to show that $\Gamma$ acts properly discontinuously on $B$. Equivalently, the map
$$\Gamma \times B \rightarrow B \times B, \quad (\gamma,x) \mapsto (x, \gamma x)$$
is proper (that is, the inverse image of a compact set is compact). Equivalently, if $(\gamma_n,x_n)$ is a sequence in $\Gamma \times B$ such that $x_n \rightarrow x$ and $\gamma_nx_n \rightarrow y$ as $n \rightarrow \infty$ for some $x,y \in B$, then $\gamma_n=\gamma$, for some $\gamma \in \Gamma$, 
for infinitely many $n$.
Let $x_n=P_n+N_n$, $x=P+N$, and $y=P'+N'$ be the Zariski decompositions of $x_n$, $x$, and $y$. Then $P_n \rightarrow P$ and $N_n \rightarrow N$ as $n \rightarrow \infty$ by continuity of the Zariski decomposition on the interior of the effective cone \cite{BKS04}, Proposition~1.16. Also, since by assumption $\Gamma$ preserves the semigroup of effective classes, $\gamma_nx_n=\gamma_nP_n+\gamma_nN_n$ is the Zariski decomposition of $\gamma_nx_n$. Thus $\gamma_nP_n \rightarrow P'$ and $\gamma_nN_n \rightarrow N'$ as $n \rightarrow \infty$. Now $P_n \rightarrow P$ and $\gamma_nP_n \rightarrow P'$ implies $\gamma_n=\gamma$, some $\gamma \in \Gamma$, for infinitely many $n$ because $\Gamma$ acts properly discontinuously on $C^+$.
\end{proof}

\begin{lemma} \label{monodromy_preserves_effective}
Let $(Y,D)$ be a Looijenga pair. Let $(Y_g,D)$ be a generic deformation of $(Y,D)$. Then the monodromy group $\Adm_{Y}$ preserves the semigroup of effective classes on $(Y_g,D)$.
\end{lemma}
\begin{proof}
By Lemma~\ref{Moricones}(1), if $C \subset Y_g$ is an irreducible curve, then either $C \subset D$, $C^2 \ge 0$,  or $C$ is a $(-1)$-curve.
The group $\Adm_Y$ preserves the boundary classes $[D_i]$  and the ample cone of $Y_g$ by Lemma~\ref{Nefcones}.
It follows from Riemann--Roch that $\mu(C)$ is effective for $\mu \in \Adm_Y$ and $C$ either an irreducible curve such that $C^2 \ge 0$ or a $(-1)$-curve.
So $\Adm_Y$ preserves the semigroup of effective classes.
\end{proof}

\begin{corollary}\label{properly_discontinuous_on_generalized_Tits_cone}
Let $(Y,D)$ be a Looijenga pair. Let $(Y_g,D)$ be a generic deformation of $(Y,D)$.
Then $\Adm_{Y}$ acts properly discontinuously on the interior of $\overline{\NE}(Y_g)+\langle D_1,\ldots,D_n\rangle_{\bR}$.
\end{corollary}
\begin{proof}
The group $\Adm_{Y}$ acts properly discontinuously on the interior of $\overline{\NE}(Y_g)$ by Lemma~\ref{monodromy_preserves_effective} and Proposition~\ref{discontinuous_action_Mori_cone}. 
Since $\Adm_{Y}$ acts trivially on the subspace $\langle D_1,\ldots,D_n\rangle_{\bR}$, it follows that $\Adm_{Y}$ acts properly discontinuously on the interior of $\overline{\NE}(Y_g)+\langle D_1,\ldots,D_n\rangle_{\bR}$. 
(Indeed any two points $x,y$ in the interior of $\overline{\NE}(Y_g)+\langle D_1,\ldots,D_n \rangle_{\bR}$ are contained in a translate $T$ of the interior of $\overline{\NE}(Y_g)$ by some element $z=\sum a_i D_i$, $a_i \in \bR$. 
Thus there exist open neighbourhoods $x \in U \subset T$ and $y \in V \subset T$ such that the set $\{ g \in \Adm_{Y} \ | \ gU \cap V \neq \emptyset \}$ is finite
because $\Adm_{Y}$ acts properly discontinuously on $T$.)
\end{proof}

\end{document}